\documentclass[11pt,oneside]{amsart} 
\usepackage{amssymb}
\usepackage{graphicx}
\usepackage{hyperref}
\usepackage[capitalize]{cleveref}
\usepackage{mathtools} 
\allowdisplaybreaks

\newcommand{\NN}{\mathbb{N}}
\newcommand{\QQ}{\mathbb{Q}}
\newcommand{\RR}{\mathbb{R}}
\newcommand{\ZZ}{\mathbb{Z}}

\newcommand{\1}{\mathbf{1}}
\newcommand{\acts}{\mathrel{\hbox{\large$\looparrowright$}}}
\newcommand{\agood}{A_{\text{\normalfont good}}}

\newcommand{\dleaf}{D^\leaf}
\newcommand{\fund}{\mathcal{F}}

\newcommand{\halfspace}{\mathcal{H}}

\newcommand{\leaf}{\text{\upshape leaf}}
\newcommand{\Lie}[1]{\mathfrak{\lowercase{#1}}}
\newcommand{\notacts}{\mathrel{\hbox{\lower1pt\hbox to 0pt{\hskip6pt$\pmb\not$\hss}\ensuremath\acts}}}
\newcommand{\open}{\mathcal{O}}

\newcommand{\tends}{\stackrel{n \to \infty}{\longrightarrow}}
\newcommand{\weyl}{\mathcal{W}}

\DeclareMathOperator{\Homeo}{Homeo}
\DeclareMathOperator{\prob}{Prob}
\DeclareMathOperator{\PSL}{PSL}
\DeclareMathOperator{\rank}{rank}
\DeclareMathOperator{\SL}{SL}

\newcommand\bigset[2]{\left\{\, #1 
 \mathrel{\left| \vphantom {\left\{ #1 \mid #2 \right\} }
 \right.} #2 \,\right\} }

\newcommand{\ul}{\underline}

\newenvironment{thmref}[1]{\begingroup\def\theequation{\ref{#1}}}{\endgroup}

\newcommand{\csee}[1]{(see~\cref{#1})}
\newcommand{\fullcsee}[2]{(see~\fullcref{#1}{#2})}

\newcommand{\pref}[1]{\textup{(\ref{#1})}}
\newcommand{\fullcref}[2]{\cref{#1}\pref{#1-#2}}
\newcommand{\fullref}[2]{\ref{#1}\pref{#1-#2}}

\numberwithin{equation}{section}

\newtheorem{cor}[equation]{Corollary}
\newtheorem{key}[equation]{Key Proposition}
\newtheorem{lem}[equation]{Lemma}
\newtheorem{prop}[equation]{Proposition}
\newtheorem{thm}[equation]{Theorem}

\crefformat{cor}{Corollary~#2#1#3}
\crefformat{key}{Key Proposition~#2#1#3}
\crefformat{lem}{Lemma~#2#1#3}
\crefformat{prop}{Proposition~#2#1#3}
\crefformat{thm}{Theorem~#2#1#3}
\crefname{cor}{corollary}{corollaries}
\crefname{key}{proposition}{propositions}
\crefname{lem}{lemma}{lemmas}
\crefname{prop}{proposition}{propositions}
\crefname{thm}{theorem}{theorems}
\Crefname{cor}{Corollary}{Corollaries}
\Crefname{key}{Proposition}{Propositions}
\Crefname{lem}{Lemma}{Lemmas}
\Crefname{prop}{Proposition}{Propositions}
\Crefname{thm}{Theorem}{Theorems}

\theoremstyle{definition}
\newtheorem*{ack}{Acknowledgments}
\newtheorem{assump}[equation]{Assumption}
\newtheorem{defn}[equation]{Definition}
\newtheorem{eg}[equation]{Example}

\newtheorem{notation}[equation]{Notation}
\newtheorem{note}[equation]{Note}
\newtheorem{recall}[equation]{Reminder}
\newtheorem{rem}[equation]{Remark}

\newtheorem{warn}[equation]{Warning}

\crefformat{recall}{Recall~#2#1#3}

\newcounter{case}
\newenvironment{case}{\par\medbreak\noindent\refstepcounter{case}\textbf{Case \thecase.}\unskip\em}{\unskip}
\newcounter{caseholder}
\numberwithin{case}{caseholder}
\renewcommand{\thecase}{\arabic{case}}

\crefformat{case}{Case~#2#1#3}
\crefname{case}{case}{cases}
\Crefname{case}{Case}{Cases}

\newcounter{subfigure}
\numberwithin{subfigure}{figure}

\makeatletter
\def\@listI{\leftmargin\leftmargini \parsep\z@skip
  \topsep\listisep \itemsep\smallskipamount 
  \listparindent\normalparindent}
 \makeatother

\makeatletter
\newcommand{\noprelistbreak}{\smallskip\@nobreaktrue\nopagebreak} 

\let\oldsection=\section
\renewcommand{\section}{\penalty-1000\oldsection}
\def\l@subsection{\@tocline{2}{0pt}{4em}{2.5pc}{}}
\def\subsection{\@startsection{subsection}{2}%
  \z@{.5\linespacing\@plus.7\linespacing\penalty-500}{-.5em}%
  {\normalfont\bfseries\mathversion{bold}}}
\makeatother

\begin{document}

\title[Non-left-orderability of lattices in semisimple Lie groups]{Non-left-orderability of lattices 
\\ in higher-rank semisimple Lie groups \\
(after Deroin and Hurtado)}

\author{Dave Witte Morris}

\address{Department of Mathematics and Computer Science,
University of Lethbridge, 
Lethbridge, Alberta, T1K~3M4, Canada}

\email{dmorris@deductivepress.ca, https://deductivepress.ca/dmorris}

{\mathversion{bold}
\maketitle
}

\begin{abstract}
Let $G$ be a connected, semisimple, real Lie group with finite centre, with $\rank_\RR G \ge 2$.
B.\,Deroin and S.\,Hurtado recently proved the 30-year-old conjecture that no irreducible lattice in~$G$ has a left-invariant total order. (Equivalently, they proved that no such lattice has a nontrivial, orientation-preserving action on the real line.) We will explain many of the main ideas of the proof, by using them to prove the analogous result for lattices in $p$-adic semisimple groups. The $p$-adic case is easier, because some of the technical issues do not arise.
\end{abstract}

\tableofcontents

\section{Introduction}

\begin{defn}[{\cite[p.~7]{DeroinNavasRivasBook}}]
A group~$\Gamma$ is \emph{left-orderable} if it has a left-invariant total order. This means there is a transitive binary relation~$\prec$, such that, for all $a,b,c \in \Gamma$, 
	\begin{itemize}
	\item exactly one of the following is true: $a \prec b$, or $a = b$, or $a \succ b$,
	and
	\item if $a \prec b$, then $ca \prec cb$.
	\end{itemize}
\end{defn}

Understanding which groups are left-orderable is a longstanding problem in algebra \cite{KopytovMedvedevBook,MuraRhemtullaBook} and is also important in topology \cite{ClayRolfsenBook}. The following result proves a conjecture that was suggested informally (at least for the case where $G$ is simple) by the current author more than 30 years ago, and became official when it was proposed independently by \'E.\,Ghys \cite[p.~200]{Ghys-ActRes} in 1999.

\begin{thm}[Deroin-Hurtado {\cite[Thm.~1.3]{DeroinHurtado}}] \label{DeroinHurtadoThm}
If\/ $\Gamma$ is an irreducible lattice in a connected, semisimple, real Lie group~$G$ with finite centre, and\/ $\rank_\RR G \ge 2$, then $\Gamma$ is not left-orderable. 
\end{thm}

\Cref{DeroinHurtadoThm} can be restated as follows:

\begin{cor}[Deroin-Hurtado {\cite[Thm.~1.3]{DeroinHurtado}}] \label{NoActR}
Assume\/ $\Gamma$ is an irreducible lattice in a connected, semisimple, real Lie group~$G$ with finite centre, and\/ $\rank_\RR G \ge 2$. Then\/ $\Gamma$ has no nontrivial action on the real line by orientation-preserving homeomorphisms.
\end{cor}

\begin{proof}
Suppose $\Gamma$ has a nontrivial, orientation-preserving action on~$\RR$, and let $N$ be the kernel of the action. Then $\Gamma/N$ has a faithful, orientation-preserving action on~$\RR$, so it is not difficult to see that $\Gamma/N$ is left-orderable \cite[Thm.~6.8]{Ghys-ActOnCircle}.

On the other hand, since no finite groups are left-orderable \cite[\S1.4.1, p.~27]{DeroinNavasRivasBook}, we know that $N$ is a normal subgroup of infinite index, so the Margulis Normal Subgroup Theorem \cite[Thm.~17.1.1, p.~347]{Morris-ArithGrpsBook} tells us that $N$ is contained in the centre of~$G$. So $N$ is normal in~$G$. Then $\Gamma/N$ is a left-orderable, irreducible lattice in the semisimple Lie group $G/N$. This contradicts \cref{DeroinHurtadoThm}.
\end{proof}

\begin{rem} \label{VirtTrivR}
If we do not require the homeomorphisms to be orientation-preserving, then there could be a nontrivial action: the map $x \mapsto -x$ is a homeomorphism of order~$2$, and it is possible that $\Gamma$ has a nontrivial homomorphism to~$\ZZ/2\ZZ$. However, it follows from \cref{DeroinHurtadoThm} that every action of~$\Gamma$ on~$\RR$ (whether orientation-preserving or not) must be \emph{virtually trivial}. This means that some finite-index subgroup of~$\Gamma$ acts trivially. (In fact, the subgroup will have index $1$ or~$2$.)
\end{rem}

By combining \cref{DeroinHurtadoThm} with the following result, Deroin and Hurtado also obtained a strong result for actions on the circle~$S^1$, not only actions on~$\RR$.

\begin{thm}[Ghys \cite{Ghys-ActRes}, cf.\ Burger-Monod {\cite[Cor.~1.5]{BurgerMonod-BddCoho}, \cite[Cor.~1.7]{Burger-criterion}}] \label{GhysBurgerMonod}
Assume\/ $\Gamma$ is an irreducible lattice in a connected, semisimple, real Lie group~$G$, and\/ $\rank_\RR G \ge 2$. If no almost-simple factor of~$G$ is locally isomorphic to\/ $\SL(2,\RR)$, then every action of\/~$\Gamma$ on the circle~$S^1$ has a finite orbit.
\end{thm}

\begin{cor}[Deroin-Hurtado \cite{DeroinHurtado}] \label{NoActCircle}
Assume\/ $\Gamma$ is an irreducible lattice in a connected, semisimple, real Lie group~$G$, and\/ $\rank_\RR G \ge 2$. If no almost-simple factor of~$G$ is locally isomorphic to\/ $\SL(2,\RR)$, then every action of\/ $\Gamma$ on the circle~$S^1$ is virtually trivial.
\end{cor}

\begin{proof}
By \cref{GhysBurgerMonod}, we may let $x$ be a point whose orbit is finite. Then its stabilizer $\Gamma_{\!x}$ is a finite-index subgroup that acts on $S^1 \smallsetminus \{x\}$, which is homeomorphic to~$\RR$. Then we see from \cref{VirtTrivR} that a finite-index subgroup of~$\Gamma_{\!x}$ must act trivially.
\end{proof}

\begin{rem}
The group $\PSL(2,\RR)$ acts faithfully on $\RR \cup \{\infty\} \simeq S^1$ (by linear-fractional transformations). Therefore, if some almost-simple factor of~$G$ is locally isomorphic to $\SL(2,\RR)$, then every lattice in~$G$ has a highly nontrivial action on~$S^1$ (because there is a surjective homomorphism $G \to \PSL(2,\RR)$). This shows that \cref{NoActCircle} would not be valid without the hypothesis that no factor is locally isomorphic to $\SL(2,\RR)$.
\end{rem}

For simplicity, the statements of the above results assume that the centre of~$G$ is finite. There are counterexamples without this assumption, but Deroin and Hurtado proved that all of them arise from the fact that the action of $\PSL(2,\RR)$ on~$S^1$ lifts to an orientation-preserving action of $\widetilde\SL(2,\RR)$ (the universal cover of $\SL(2,\RR)$) on~$\RR$. In particular:

\begin{cor}[Deroin-Hurtado {\cite[Thm.~1.1 and \S2]{DeroinHurtado}}]
Assume $G$ is a connected, semisimple, real Lie group, and\/ $\rank_\RR G \ge 2$. 
An irreducible lattice~$\Gamma$ in~$G$ is left-orderable if and only if it is torsion-free and there is a surjective homomorphism $G \to \widetilde\SL(2,\RR)$.
\end{cor}

\begin{rem}
\Cref{DeroinHurtadoThm} is a huge advance over what was previously known. In particular:
	\begin{enumerate}
	\item There was not a single example of a cocompact, torsion-free lattice~$\Gamma$ (in a connected, semisimple Lie group), for which it was known that no finite-index subgroup of~$\Gamma$ is left-orderable.
	\item For the noncocompact case, the theorem had been proved under the additional assumption that either $\rank_\QQ \Gamma \ge 2$ \cite{Witte-Circle} or $G$ has more than one noncompact simple factor \cite{LifschitzMorris}. However, the theorem was not known to be true even for the special case where $G = \SL(3,\RR)$.
	\end{enumerate}
See \cite{Ghys-ActOnCircle,Morris-CanLattAct} for expository discussions of actions of lattices (and related groups) on~$\RR$ or~$S^1$, and see \cite[Chap.~4]{DeroinNavasRivasBook} for an introduction to the viewpoint on (non-)left-orderability that is used in the proof of \cref{DeroinHurtadoThm}.
\end{rem}

\begin{rem}
The circle is a $1$-dimensional manifold, so the Deroin-Hurtado Theorem is a part of the ``Zimmer program\rlap,'' which includes the study of actions of lattices on low-dimensional manifolds. (However, unlike in \cref{DeroinHurtadoThm}, the actions are usually assumed to be differentiable, and are often~$C^\infty$.) The recent work of A.\,Brown, D.\,Fisher, and S.\,Hurtado \cite{BrownFisherHurtado-SL(mZ),BrownFisherHurtado-subexp,BrownFisherHurtado-nonuniform} is a huge advance in the area (and seems to have been part of the inspiration for Deroin and Hurtado). See \cite{Fisher-ActOnMfld} for a survey of the Zimmer program, and see \cite{Cantat-BFH,Fisher-afterward} for expositions of the first Brown-Fisher-Hurtado paper.
\end{rem}

In this paper, we will explain some of the main ideas in the proof of \cref{DeroinHurtadoThm} (or, equivalently, \cref{NoActR}). However, $p$-adic groups are easier than real ones, so we will actually prove the following:

\begin{thm}[Deroin-Hurtado] \label{padicThm}
If\/ $\Gamma$ is an irreducible lattice in a $p$-adic semisimple Lie group~$G$, and\/ $\rank_{\QQ_p} G \ge 2$, then\/ $\Gamma$ has no nontrivial action on the real line by orientation-preserving homeomorphisms. 
\end{thm}

Here is an outline of the paper.
	\begin{itemize}
	\item \Cref{PrelimSect} recalls some standard definitions and notation, and a few basic results. (The reader is encouraged to skip this \lcnamecref{PrelimSect}, and refer back when necessary.) 
	\item \Cref{OverviewSect} sketches the proof of \cref{padicThm}. 
	\item Further explanation of several arguments can be found in \cref{DeroinHurtadoPfSect}. Some details are omitted, but specific references to the original work of Deroin and Hurtado \cite{DeroinHurtado} are provided.
	\item Our exposition of the proof of \cref{padicThm} relies on the simplifying assumption that $G = K \Gamma$ \csee{G=KGamma}, but \cref{FiniteSetSect} provides a brief introduction to the main tool that allows this assumption to be eliminated (namely, the theory of harmonic functions).
	\end{itemize}

\begin{note}
This is an expository paper. 
None of the nontrivial ideas or results are new, so the author does not claim credit for any of them, even if no reference is given. In particular, although \cref{padicThm} is not stated in~\cite{DeroinHurtado}, it follows directly from the methods there, so it can be considered to be implicit in the paper and should be attributed to Deroin and Hurtado.
 
On the other hand, the author takes responsibility for any errors and other deficiencies in the manuscript (of course).
\end{note}

\begin{ack}
I thank B.\,Deroin, S.\,Hurtado, and J.\,L\'ecureux for informative discussions related to the proof of \cref{DeroinHurtadoThm} and possible generalizations. I also acknowledge support of the Institut Henri Poincar\'e (UAR 839 CNRS-Sorbonne Universit\'e) and LabEx CARMIN (ANR-10-LABX-59-01).
\end{ack}

\section{Preliminaries} \label{PrelimSect}

\subsection{Measure theory}

\begin{recall} \label{RecallMeasureTheory} 
Let $\mu$ be a measure on a separable, metrizable topological space~$Y$. 
	\begin{enumerate}
	\item We say that $\mu$ is a  \emph{probability measure} if $\mu(Y)= 1$.
	\item The \emph{support} of~$\mu$ is the (unique) smallest closed set whose complement has measure~$0$.
	\item $\mu$ is \emph{atomic} if $\mu(B) = \sum_{b \in B} \mu \bigl( \{b\} \bigr)$ for every Borel set~$B$. Informally, this means that $\mu$ is a sum of point masses.
	\item A property is true \emph{for $\mu$-a.e.\ $y \in Y$} if the set of points where it is false has measure~$0$. In this case, we may also say that the property is true \emph{$\mu$-a.e.}, or that is is true for $\mu$-a.e.~$y \in Y$.
	\item A measure $\mu'$ is \emph{absolutely continuous} with respect to~$\mu$ if, for every Borel set~$B$, such that $\mu(B) = 0$, we also have $\mu'(B) = 0$.
	\item If $\mu'$ is absolutely continuous with respect to~$\mu$, then there is a non-negative measurable function $d\mu'/d\mu$ (called the \emph{Radon-Nikodym derivative} of~$\mu'$ with respect to~$\mu$), such that $\mu' = (d\mu'/d\mu) \, \mu$, in the sense that, for every real-valued continuous function with compact support, we have
		\[ \int_Y f \, d\mu' = \int_Y f(y) \, \frac{d\mu'}{d\mu}(y) \, d \mu(y) . \]
	The Radon-Nikodym derivative is unique up to a set of $\mu$-measure~$0$.
	\item If two measures are absolutely continuous with respect to each other, then they are in the same \emph{measure class}. (This is an equivalence relation.)
	\item  \label{RecallMeasureTheory-f*}
	If $f \colon Y \to Z$ is a continuous function to a topological space~$Z$, then we define a measure $f_*\mu$ on~$Z$ by
		\[ f_*\mu(B) = \mu \bigl( f^{-1}(B) \bigr) . \]
	\end{enumerate}
\end{recall}

\subsection{Topological groups}

\begin{recall}
Assume $H$ is a locally compact topological group.
\leavevmode
	\begin{itemize}
	\item $H$ is said to be \emph{locally compact} if it has a nonempty, open subset whose closure is compact. All groups in this paper (including $G$, $K$, $\Gamma$, and~$\RR$) are locally compact (and second countable).
	\item It is a basic fact in the theory of topological groups that every locally compact group has a \emph{Haar measure}. This is a measure $m_H$ on~$H$, such that:
		\begin{itemize}
		\item $m_H$ is left-invariant: $m_H(h B) = m_H(B)$ for all $h \in H$ and every Borel set~$B$,
		\item $m_H$ is locally finite: $m_H(C) < \infty$ for every compact set~$C$,
		and
		\item $m_H$ has full support: $m_H(U) > 0$ for every open set~$U$.
		\end{itemize}
	Furthermore, the Haar measure is unique, up to a scalar multiple: if $\mu$ is another Haar measure on~$H$, then there is a constant $c > 0$, such that $\mu(B) = c \, m_H(B)$ for every Borel set~$B$.
	\item $H$ is \emph{unimodular} if its Haar measure is bi-invariant. This means it is both left-invariant and right-invariant: $\mu_H(hB) = \mu_H(B) = \mu_H(Bh)$. Note that this implies that $\mu_H$ is invariant under conjugation: $\mu_H(h B h^{-1}) = \mu_H(B)$.
	\item If $H$ is either a semisimple Lie group or a discrete group or a compact group, then $H$ is unimodular.
	\item If $H$ is compact, then $m_H(H) < \infty$ (because the measure of every compact set is finite), and we usually choose the normalizing scalar to make $m_H$ a probability measure. 
	\item If $H$ is \ul{not} compact, then $m_H(H) = \infty$, so it is not possible for the Haar measure on a noncompact group to be a probability measure.
	\end{itemize}
\end{recall}

\begin{defn}
Assume $G$ is a unimodular topological group.
\noprelistbreak
	\begin{itemize}
	\item If $\Gamma$ is any discrete subgroup of~$G$, then any Haar measure $m_G$ induces a well-defined $G$-invariant measure $m_{G/\Gamma}$ on $G/\Gamma$ by setting:
		\[ m_{G/\Gamma}(B) = m_G \bigl( \pi^{-1}(B) \cap \fund \bigr) , \]
where $\pi \colon G \to G/\Gamma$ is the natural quotient map ($\pi(g) = g \Gamma$), and $\fund$ is a measurable set of coset representatives for~$\Gamma$ in~$G$ (or, in other words, $\fund$ is a \emph{fundamental domain} for $\Gamma$ in~$G$).
	\item A subgroup~$\Gamma$ of a locally compact topological group~$G$ is a \emph{lattice} in~$G$ if 
	$\Gamma$ is discrete, 
	and
	the quotient space $G/\Gamma$ has finite measure (i.e., $m_{G/\Gamma}(G/\Gamma) < \infty$).
	\item A lattice $\Gamma$ is \emph{irreducible} if $\Gamma N$ is dense in~$G$, for every closed, noncompact, normal subgroup~$N$ of~$G$.
	\item A closed subgroup~$\Gamma$ of~$G$ is \emph{cocompact} if $G/\Gamma$ is compact. 
	\end{itemize}
\end{defn}

Every discrete, cocompact subgroup is a lattice (because the Haar measure of every compact set is finite). The converse is not true in general, but the first part of the following result is a special case of the fact that it is true in $p$-adic groups:

\begin{thm} \label{SSLattFacts} 
Assume $\Gamma$ is an irreducible lattice in a $p$-adic semisimple Lie group~$G$, such that $\rank_{\QQ_p} G \ge 2$. Then:
	\begin{enumerate}
	\item  \label{SSLattFacts-padicCocpct}
	$\Gamma$ is cocompact \cite[Prop.~IX.3.7, p.~313]{MargulisBook},
	\item \label{SSLattFacts-abel}
	the commutator subgroup $[\Gamma, \Gamma]$ has finite index in~$\Gamma$ \cite[Thm.~IX.5.4, p.~325]{MargulisBook},
	and 
	\item \label{SSLattFacts-CircleAction}
	if~$\Gamma$ acts on the circle~$S^1$ by homeomorphisms, then the action has a finite orbit \cite[Cor.~6.10]{WitteZimmer-CircleBundle}.
	\end{enumerate}
\end{thm}

\subsection{Group actions}

\begin{defn} \label{GrpActDefns}
Assume
	\begin{itemize}
	\item a topological group~$H$ acts by homeomorphisms on a metrizable topological space~$Y$, 
	\item $\mu_H$ is a probability measure on~$H$,
	and
	\item $\mu_Y$ is a probability measure on~$Y$.
	\end{itemize}
Then:
	\begin{enumerate}
	\item The action of~$H$ is \emph{free} if $h y \neq y$ for all nonidentity $h \in H$ and all $y \in Y$.
	\item The measure $ \mu_Y$ is \emph{$H$-invariant} if $h_* \mu_Y =  \mu_Y$ for all $h \in H$ (where $h_*\mu_Y$ was defined in \fullcref{RecallMeasureTheory}{f*}). In other words, $\mu_Y(hB) =  \mu_Y(B)$, for every $h \in H$ and every Borel set~$B$.
	\item An $H$-invariant measure~$\mu_Y$ is \emph{ergodic} if, for every $H$-invariant Borel set~$B$, either $\mu_Y(B) = 0$, or $\mu_Y(Y \setminus B) = 0$.
	\item The convolution $\mu_H * \mu_Y$ is the measure on~$Y$ defined by 
		\[ \text{$\mu_H * \mu_Y = \int_H h_* \mu_Y \, d\mu_H(H)$,
		\quad so \quad 
		$(\mu_H * \mu_Y)(B) = \int_H \mu_Y(h^{-1} B) \, d\mu_H(h)$} . \]
	\item \label{GrpActDefns-stationary}
	We say that $\mu_Y$ is \emph{$\mu_H$-stationary} if $\mu_H * \mu_Y = \mu_Y$. More concretely, for every Borel set~$B$, this means that we have
	\[ \mu_Y(B) = \int_H \mu_Y(h^{-1} B) \, d\mu_H(h) . \]
In other words, the measure of every set is equal to the average of the measures of its translates.

	\item If $\mu_Y$ is $H$-invariant, then there is a unique probability measure~$\eta$ on the space $\mathcal{E}$ of ergodic $H$-invariant measures on~$Y$, such that
	\[ \mu_Y = \int_{\mathcal{E}} \xi \, d\eta(\xi) . \]
(See \cite[Thm., p.~77]{Phelps-Choquet}, for example.)
Measures in the support of~$\eta$ are called \emph{ergodic components} of~$\mu_Y$. 
	\item A function $f \colon \RR \to \RR$ is \emph{Lipschitz} if there is a constant $C \in \RR$, such that
		\[ \text{$|f(x) - f(y)| \le C \, |x - y|$ \quad for all $x,y \in \RR$} . \]
	\end{enumerate}
\end{defn}

\subsection{Dynamical systems}

\begin{prop}[cf.\ {\cite[Cor.~4.2 and Thm.~4.4, pp.~98--99]{EinsiedlerWardBook}}] \label{ExistsRInvProb}
Assume $\RR$ acts continuously on a compact, metrizable space~$Y$. Then there is at least one\/ $\RR$-invariant probability measure on~$Y$ that is ergodic.
\end{prop}

\begin{thm}[Pointwise Ergodic Theorem {\cite[Thm.~2.30, p.~44]{EinsiedlerWardBook}}] \label{PointwiseErgThm}
Assume a cyclic group $\langle g \rangle$ acts by homeomorphisms on a metrizable space~$Y$, and $\mu$ is an ergodic $\langle g \rangle$-invariant probability measure on~$Y$. Then, for any continuous function $f$ on~$Y$ with compact support, and $\mu$-a.e.\ $y \in Y$, we have
	\[ \lim_{n \to \infty} \frac{1}{n} \sum_{i=1}^n f(g^i y) = \int_Y f \, d\mu . \]
\end{thm}

\begin{prop} \label{PActsMinimal} 
Let $\Gamma$ be a \ul{cocompact}, irreducible lattice in a \textup(real or $p$-adic\textup) semisimple Lie group~$G$, and let $P$ be a parabolic subgroup of~$G$. Then:
	\begin{enumerate}
	\item \label{PActsMinimal-meas}
	The $P$-invariant probability measure on $G/\Gamma$ is unique \textup(namely, the $G$-invariant measure $m_{G/\Gamma}$ is the only $P$-invariant probability measure\textup).
	\item \label{PActsMinimal-set}
	$G/\Gamma$ has no closed, nonempty, proper, $P$-invariant subsets. In other words, for all $y \in G/\Gamma$, the orbit $Py$ is dense in $G/\Gamma$.
	\end{enumerate}
\end{prop}

\begin{proof}[Idea of proof]
Since $P$ is parabolic, we know there is an $\RR$-split semisimple element $a \in P$, such that
	\[ P = \bigl\{\, p \in G \mid \text{$\{\, a^n p a^{-n} \mid n \le 0 \,\}$ is bounded} \, \bigr\} . \]
By the Moore Ergodicity Theorem \cite[Exer.~11.2\#11, p.~218]{Morris-ArithGrpsBook}, we also know that $a$ is ergodic on $G/\Gamma$ (with respect to the $G$-invariant probability measure $m_{G/\Gamma}$).

\pref{PActsMinimal-set} Since $a$ is ergodic, almost every $a$-orbit is dense. Hence, we can choose some $y' \approx y$, such that $\{\, a^n y' \mid n \ge 0\,\}$ is dense in $G/\Gamma$.
Let
	\[ U_a^- = \bigl\{\, u \in G \mid a^n u a^{-n}  \tends \1 \,\} . \]
Then $U_a^- P$ contains a neighbourhood of~$\1$ in~$G$, so we may write $y' = u p y$ with $u \in U_a^-$ and $p \in P$.

For any $x \in G/\Gamma$, there is some large $n > 0$, such that $a^n y' \approx x$. Then
	\[ x \approx a^n y' =  a^n u p y = a^n u^- a^{-n} \cdot a^n p y \approx \1 \cdot a^n p y \in P y . \]

\pref{PActsMinimal-meas} Let $\mu$ be an ergodic $P$-invariant probability measure on $G/\Gamma$, let $f$ be a continuous function on $G/\Gamma$, and let
	\[ Y = \bigset{ y \in G/\Gamma }{ \frac{1}{n} \sum_{i=1}^n f(a^i y) \tends \int_{G/\Gamma} f \, d\mu } . \]
We have $\mu(Y) = 1$ by the Pointwise Ergodic Theorem~\pref{PointwiseErgThm}. Since $\mu$ is $P$-invariant, this implies that for $\mu$-a.e.\ $u \in Y$ and $m_P$-a.e.\ $p \in P$, we have $py \in Y$. Also, it is not difficult to see that $uy \in Y$ for every $u \in U_a^-$. This easily implies that $m_{G/\Gamma}(Y) > 0$. Hence, by the Pointwise Ergodic Theorem~\pref{PointwiseErgThm}, there exists $y \in Y$, such that
	\[ \frac{1}{n} \sum_{i=1}^n f(a^i y) \tends \int_{G/\Gamma} f \, dm_{G/\Gamma} . \]
So we must have $\int_{G/\Gamma} f \, d\mu = \int_{G/\Gamma} f \, dm_{G/\Gamma}$.
\end{proof}

\section{Main ideas of the Deroin-Hurtado proof} \label{OverviewSect}

In this \lcnamecref{OverviewSect}, we will present many of the main ideas in the Deroin-Hurtado paper \cite{DeroinHurtado}. (For readers interested in more details, references to specific statements in~\cite{DeroinHurtado} will usually be provided. In many cases, additional details are also available in \cref{DeroinHurtadoPfSect}.)
Our goal is to establish \cref{padicThm}. For the reader's convenience, we reproduce the statement of this result:

\begin{thmref}{padicThm}
\begin{thm}[Deroin-Hurtado] 
If\/ $\Gamma$ is an irreducible lattice in a $p$-adic semisimple Lie group~$G$, and\/ $\rank_{\QQ_p} G \ge 2$, then\/ $\Gamma$ has no nontrivial action on the real line by orientation-preserving homeomorphisms. 
\end{thm}
\end{thmref}

\subsection{Assumptions and notation}

\begin{notation} \label{pGGamma}
Throughout the remainder of this paper:
	\begin{itemize}
	\item $p$ is a prime number,
	\item $G$ is a $p$-adic semisimple Lie group, such that $\rank_{\QQ_p} G \ge 2$,
	\item $\Gamma$ is an irreducible lattice in~$G$,
	\item $K$ is a compact, open subgroup of~$G$,
	\item $P$ is a minimal parabolic subgroup of~$G$,
	and
	\item $A$ is a maximal $\QQ_p$-split torus in~$P$.
	\end{itemize}
\end{notation}

The proof of \cref{padicThm} is by contradiction, so we assume:

\begin{assump} \label{AssumeActsOnR}
 $\Gamma$ has a faithful action by orientation-preserving homeomorphisms of~$\RR$.
\end{assump}

\begin{rem} \label{G=SL3}
\leavevmode\noprelistbreak
	\begin{enumerate}
	\item The proof of \cref{NoActR} shows that any nontrivial, orientation-preserving action of~$\Gamma$ on~$\RR$ will become faithful after modding out a finite, normal subgroup of~$G$. Therefore, even though \cref{padicThm} refers to a \emph{nontrivial} action, no loss of generality results from \cref{AssumeActsOnR}'s requirement that the action is \emph{faithful}.
	\item The reader will not miss out on any of the main ideas if they assume that $G = \SL(3, \QQ_p)$ or that $G = \SL(2, \QQ_p) \times \SL(2,\QQ_p)$. Even these two special cases were not known before the work of Deroin and Hurtado.
	\item \label{G=SL3-P}
	If $G = \SL(3,\QQ_p)$, we may let
	\[ K = \SL(3, \ZZ_p),
	\qquad
	P =  \left[ \begin{matrix} * & * & * \\ 0 & * & * \\ 0 & 0 & * \end{matrix} \right] ,
	\qquad
	A =  \left[ \begin{matrix} * & 0 & 0 \\ 0 & * & 0 \\ 0 & 0 & * \end{matrix} \right] 
	. \]
	\item The same proof applies when $G$ is a finite product of $p$-adic semisimple Lie groups, for various primes~$p$.
	\end{enumerate}
\end{rem}

\begin{note} \label{KGGammaFinite}
By \fullcref{SSLattFacts}{padicCocpct}, we know that $G/\Gamma$ is compact.
On the other hand, since $K$ is open, we also know that $K \backslash G$ is discrete. Therefore, the double-coset space $K \backslash G / \Gamma$ is both compact and discrete. So $K \backslash G / \Gamma$ must be finite. 
\end{note}

\begin{assump} \label{G=KGamma}
To simplify the proof, we will assume that the finite set $K \backslash G / \Gamma$ has only one element. This means:
	\begin{align} \label{G=KGamma-eqn}
	G = K \, \Gamma 
	. \end{align}
(See \cref{FiniteSetSect} for some comments on the more complicated situation where $K \backslash G / \Gamma$ is a larger finite set.)
\end{assump}

\begin{rem}
The conclusion of \cref{KGGammaFinite} is a key reason why the $p$-adic case is much easier than the real case. If $G$ is a (noncompact) real semisimple Lie group  and $K$ is a compact subgroup of~$G$, then $K \backslash G / \Gamma$ is a highly nontrivial manifold that plays a role in the proof of the Deroin-Hurtado Theorem~\pref{DeroinHurtadoThm}.
\end{rem} 

\begin{note} \label{G=KxGamma} 
It is not difficult to see that the real line has no nontrivial, orientation-preserving homeomorphisms of finite order. Hence, \cref{AssumeActsOnR} implies that $\Gamma$ has no nontrivial elements of finite order \cite[\S1.4.1, p.~27]{DeroinNavasRivasBook}. Since every subgroup of~$\Gamma$ is discrete, and every discrete subgroup of a compact group is finite, this implies $\Gamma \cap K$ is trivial.
By combining this with~\pref{G=KGamma-eqn}, we see that 
	\begin{align} \label{G=KxGamma-eqn} 
	\text{the natural map $(k, \gamma) \mapsto k \gamma $ is a $K$-equivariant homeomorphism $K \times \Gamma \simeq G$} 
	.\end{align}
Also, we can identify $G/\Gamma$ with~$K$. 
\end{note}

Recall that a Borel measure~$\mu$ on a topological space~$Y$ is a \emph{probability measure} if $\mu(Y) = 1$.

\begin{notation} \label{nuGDefn}
Let 
\noprelistbreak
	\begin{enumerate}
	\item $m_K$ be the Haar measure on~$K$ (normalized to be a probability measure),
	and
	\item \label{nuGDefn-G}
	$\mu_G$ be a nice probability measure on~$G$ whose support is compact and generates~$G$. We assume \cite[Defn.~3.6]{DeroinHurtado}:
		\begin{enumerate}
		\item $\mu_G$ is absolutely continuous with respect to the Haar measure~$m_G$,
		\item \label{nuGDefn-G-Kinvt}
		$\mu_G$ is \emph{bi-$K$-invariant}, which means $\mu_G(B) = \mu_G(kB) = \mu_G(Bk)$ for every Borel set~$B$ and every $k \in K$,
		and
		\item $\mu_G$ is \emph{symmetric}, which means $\mu_G(B^{-1}) = \mu_G(B)$ for every Borel set~$B$, where $B^{-1} = \{\, b^{-1} \mid b \in B\,\}$.
		\end{enumerate}
	\end{enumerate}
\end{notation}

\begin{warn}
The measure~$\mu_G$ is definitely \emph{not} the Haar measure on~$G$ (because the Haar measure on a noncompact group can never be a probability measure and cannot have compact support).
\end{warn}

\begin{defn} \label{nuGammaDefn}
Define an atomic probability measure~$\mu_\Gamma$ on~$\Gamma$ by $\mu_\Gamma \bigl( \{\gamma\} \bigr) = \mu_G(K\gamma)$. (Note that the support of~$\mu_\Gamma$ is finite, since the support of~$\mu_G$ is compact.)
Since $\mu_G$ is $K$-invariant on the left (and the Haar measure on~$K$ is unique), this implies 
	\[ \mu_G = m_K \times \mu_\Gamma , \]
where we are using the identification $G \simeq K \times \Gamma$ \csee{G=KxGamma}. We assume $\mu_G$ is chosen so that the support of $\mu_\Gamma$ generates~$\Gamma$.
\end{defn}

\begin{rem}
It is more difficult to construct the appropriate measure~$\mu_\Gamma$ when \cref{G=KGamma} does not hold \fullcsee{RandWalkDefns}{muGamma}, and it is usually not true that $\mu_\Gamma$ has finite support. 
\end{rem}

\subsection{The almost-periodic space \texorpdfstring{$Z$}{Z} and the induced \texorpdfstring{$G$}{Z}-space \texorpdfstring{$X$}{X}}

The action of~$\Gamma$ on~$\RR$ obviously extends to an action of~$\Gamma$ on the one-point compactification $\RR \cup \{\infty\}$. However, $\Gamma$ fixes the point~$\infty$ in this action. The following \lcnamecref{AlmPerSpace} shows that $\Gamma$ acts with no fixed points on a much more sophisticated compactification of~$\RR$. (This is a general result about finitely generated groups that act on~$\RR$; it does not use our assumption that $\Gamma$ is a lattice.) Deroin and Hurtado \cite[\S1.3.5]{DeroinHurtado} call this compactification the \emph{almost-periodic space}.

\begin{thm}[see \cref{AlmPerSpacePf}, {\cite[Thm.~1.1]{Deroin-AlmPer}, \cite[Thm.~5.4]{DeroinHurtado}, \cite[Cor.~4.2.11]{DeroinNavasRivasBook}}] \label{AlmPerSpace}
There is a \textup(nonempty\textup) compact metrizable space $Z$, such that: 
	\begin{enumerate}
	\item \label{AlmPerSpace-act}
	$\Gamma$ acts on~$Z$ \textup(by homeomorphisms\textup) with no global fixed point, 
	\item \label{AlmPerSpace-R}
	$\RR$ has a continuous action on $Z$ that is free, 
	\item \label{AlmPerSpace-orbits}
	each\/ $\RR$-orbit is\/ $\Gamma$-invariant, 
	and 
	\item \label{AlmPerSpace-orpresLip}
	the action of\/~$\Gamma$ on each\/ $\RR$-orbit is by orientation-preserving Lipschitz maps.
	\end{enumerate}
\end{thm}

\begin{notation}
We will use additive notation for the action of~$\RR$ on~$Z$:
	\[ \text{the image of a point $z \in Z$ by $t \in \RR$ is denoted $z + t$} . \]
This additive notation allows us to also define subtraction for points that are in the same $\RR$-orbit. Namely, for $z_1, z_2 \in Z$ and $t \in \RR$: 
	\[ \text{if $z_1 = z_2 + t$, then we let $z_1 - z_2 \coloneqq t$} . \]
Since the $\RR$-action is assumed to be free, the difference $z_1 - z_2$ is well-defined (when it exists, i.e., whenever $z_1$ and~$z_2$ are in the same $\RR$-orbit).
\end{notation}

We ``induce'' (or ``suspend'') the $\Gamma$-action to obtain a $G$-action (on a different space).

\begin{defn}[{\cite[\S5.3]{DeroinHurtado}, \cite[p.~75]{ZimmerBook}}] \label{XDefn}
Let 
	\[ \text{$X = \mathop{\mathrm{Ind}}\nolimits_\Gamma^G Z \coloneqq (G \times Z) / \Gamma$, \quad 
	where \quad $(h,z)*\gamma = (h\gamma, \gamma^{-1} z)$} . \] 
Then $G$ acts on this quotient space by 
	\[ g [(h,z)] = {[(gh, z)]} , \]
where we use $[(h,z)]$ to denote the image of $(h,z)$ under the natural quotient map from $G \times Z$ to~$(G \times Z)/\Gamma$.
\end{defn}

\begin{note} \label{Xbundle}
\leavevmode\noprelistbreak
	\begin{enumerate}
	\item For $[(h,z)] \in X$ and $\gamma \in \Gamma$, we have
		\[ [(h\gamma,z)] = [(h\gamma \cdot \gamma^{-1}, \gamma z)] = [(h,\gamma z)] . \]
	\item \label{Xbundle-X}
	The natural map $(k, z) \mapsto [(k,z)]$ is a $K$-equivariant homeomorphism	$K \times Z \simeq X$ (by \pref{G=KxGamma-eqn}). This implies that $X$ is compact (because $K$ and $Z$ are compact).
	\item \label{Xbundle-equi}
	The function $[(h, z)] \mapsto h \Gamma$ is a well-defined $G$-equivariant map $X \to G/\Gamma$. It gives $X$ the structure of a fibre bundle over~$G/\Gamma$ with fibres homeomorphic to~$Z$.
	\item \label{Xbundle-OnLeaf}
	For $g \in G$ and $k \in K$, there exist $k' \in K$ and $\gamma = \gamma(g, k) \in \Gamma$, such that $gk = k' \gamma$ (by \pref{G=KGamma-eqn}). Then, for all $z \in Z$, we have
		\[ g [(k, z)] = [(gk, z)] = [(k' \gamma, z)] = [(k', \gamma z)] . \]
	Hence, $g$ acts on the entire fibre $[\{k\} \times Z]$ via the element~$\gamma(g, k)$.
	\end{enumerate}
\end{note}

We have defined $z_1 - z_2$ when $z_1$ and~$z_2$ are points in~$Z$ that are in the same $\RR$-orbit. We extend this to points in~$X$:

\begin{notation}
The action of~$\RR$ on~$Z$ can be extended to a free action of~$\RR$ on~$X$ via the identification $X \simeq K \times Z$: $[(k, z)] + t = [(k, z + t)]$. This yields a definition of $x - y$ for all $x,y \in X$, such that $x$ and~$y$ are in the same $\RR$-orbit.
\end{notation}

\begin{warn}
For $g \in G$, it is usually \emph{not} true that $[(g, z)] + t = [(g, z + t)]$. To see this, write $g = k\gamma$, with $k \in K$ and $\gamma \in \Gamma$ (by using~\pref{G=KGamma-eqn}). Then:
	\[ [(g, z)] + t
	= [(k\gamma, z)] + t
	= [(k, \gamma z)] + t
	= [(k, \gamma z + t)]
	= \bigl[ \bigl( k \gamma, \gamma^{-1}(\gamma z + t) \bigr) \bigr]
	= \bigl[ \bigl( g, \gamma^{-1}(\gamma z + t) \bigr) \bigr]
	. \]
This will rarely be equal to $[(g, z + t)]$, because $\Gamma$ is not assumed to commute with the $\RR$-action.
\end{warn}

However, we do have the following weaker property, which follows from \fullcref{Xbundle}{OnLeaf} (and \fullcref{AlmPerSpace}{orpresLip}):

\begin{lem}[cf.\ {\cite[\S5.3.1]{DeroinHurtado}}] \label{LipOnLeaves}
The action of~$G$ on~$X$ respects the $\RR$-orbits, and the restriction of each element of~$G$ to an $\RR$-orbit is Lipschitz \textup(with a Lipschitz constant that does not depend on the choice of the orbit\textup). More precisely, for all $g \in G$, there exists  $C = C(g) \in \RR^+$, such that, for all $x \in X$ and $t \in \RR$, we have $g(x + t) \in gx + \RR$ and
	\[ |g(x + t) - g x| \le C \, |t| . \]
\end{lem}

\subsection{The probability measures \texorpdfstring{$\mu_Z$,~$\mu_X$, and $\mu_X^P$}{on Z and X}}

\begin{notation} \label{muZmuXDefn}
Let 
	\begin{itemize}
	\item $\mu_Z$ be an ergodic $\RR$-invariant probability measure on~$Z$ (see \cref{ExistsRInvProb} for the existence of such a measure),
	and
	\item  $\mu_X = m_K \times \mu_Z$ (under the identification $X \simeq K \times Z$), so $\mu_X$ is a $K$-invariant probability measure on~$X$.
	\end{itemize}
\end{notation}

\begin{rem}
See \cref{GeneralmuX} for a definition of~$\mu_X$ when $X \not\simeq K \times Z$.
\end{rem}

Lipschitz functions on~$\RR$ are differentiable almost everywhere (and the derivative is bounded). Thus, we see from \cref{LipOnLeaves} that each element of~$G$ has a derivative along the $\RR$-orbits (a.e.):

\begin{defn}[{\cite[\S5.2.2]{DeroinHurtado}}]
For $g \in G$ and $\mu_X$-a.e.\ $x \in X$, we let
	\[ \dleaf_g(x) = \lim_{t \to 0} \frac{g(x+ t) - x}{t} . \]
We can similarly define $\dleaf_\gamma(z)$ for $\gamma \in \Gamma$ and $\mu_Z$-a.e.~$z \in Z$.
\end{defn}

By definition, the measure $\mu_Z$ is $\RR$-invariant, so moving along the $\RR$-orbits preserves the measure. Since the $\Gamma$-action preserves these $\RR$-orbits, any non-invariance of~$\mu_Z$ under this action comes from distortion within the $\RR$-orbits. Also, for any $\gamma \in \Gamma$, the action on each $\RR$-orbit is Lipschitz, so the measure class of Lebesgue measure is preserved. This implies that $\gamma_*\mu_Z$ is in the same measure class as~$\mu_Z$. Furthermore:

\begin{lem}[{\cite[Lems.~5.5 and 5.10(4)]{DeroinHurtado}}] \label{RadonNikodym}
For every $\gamma \in \Gamma$, the Radon-Nikodym derivative $d \gamma_*\mu_Z / d \mu_Z $ is equal to $\dleaf_{\gamma^{-1}}$~\textup(a.e.\textup). Similarly, for all $g \in G$, we have
	\[   \frac{d g_*\mu_X }{ d \mu_X} = \dleaf_{g^{-1}} \quad \text{\textup(a.e.\textup)} . \]
\end{lem}

The almost-periodic space~$Z$ can be constructed so that the following holds. (See \fullcref{GrpActDefns}{stationary} for the definition of ``stationary\rlap.'')

\begin{prop}[see \cref{nuGammaStationary}, {\cite[Lem.~5.6]{DeroinHurtado}}] \label{nuGammaStationaryStated}
The measure $\mu_Z$ is $\mu_\Gamma$-stationary.
\end{prop}

This has the following crucial consequence.

\begin{cor}[see \cref{muXStationary},{\cite[Prop.~5.11]{DeroinHurtado}}] \label{muXStationaryStated}
$\mu_X$ is $\mu_G$-stationary.
\end{cor}

Recall that $P$ is a minimal parabolic subgroup (see \cref{pGGamma} and perhaps also \fullcref{G=SL3}{P}).
It is known that all $\mu_G$-stationary measures come from $P$-invariant measures, by averaging the $K$-translates:

\begin{thm}[Furstenburg {\cite[Thm.~2.1]{Furstenberg-RandProd}}] \label{FurstenburgMuP} 
There is a unique $P$-invariant probability measure $\mu_X^P$ on~$X$, such that
	\[ \mu_X = \int_K k_* \, \mu_X^P \, dm_K(k) . \] 
\end{thm}

The remainder of the proof of the Deroin-Hurtado Theorem will show that the stationary measure~$\mu_X$ is $G$-invariant (or, equivalently, that $\mu_X = \mu_X^P$). This will contradict the following result, and thereby complete the proof of \cref{padicThm}, by establishing that \cref{AssumeActsOnR} must be false, because it leads to a contradiction.

\begin{prop}[{\cite[Prop.~5.12]{DeroinHurtado}}] \label{NotInvt}
We have:
\noprelistbreak
	\begin{enumerate}
	\item \label{NotInvt-Gamma}
	the stationary measure~$\mu_Z$ is \ul{not} $\Gamma$-invariant,
	\item \label{NotInvt-G}
	the stationary measure~$\mu_X$ is \ul{not} $G$-invariant,
	\item \label{NotInvt-P}
	the $P$-invariant measure~$\mu_X^P$ is \ul{not} $K$-invariant \textup(so it is not $G$-invariant\textup),
	and
	\item \label{NotInvt-=}
	$\mu_X \neq \mu_X^P$.
	\end{enumerate}
\end{prop}

\begin{proof}
\pref{NotInvt-Gamma}
Suppose $\mu_Z$ is $\Gamma$-invariant. Then for all $\gamma \in \Gamma$ we have 
	\[ \frac{d \gamma_* \mu_Z }{ d\mu_Z} = \frac{d  \mu_Z }{ d\mu_Z} = 1 \qquad \text{a.e.}, \]
so \cref{RadonNikodym} tells us $\dleaf_\gamma =1$~a.e. This implies that $\Gamma$ acts by translations on each $\RR$-orbit. However, the group of translations is abelian, and the abelianization of $\Gamma$ is finite \fullcsee{SSLattFacts}{abel}. This contradicts \cref{AssumeActsOnR}, which states that the action of~$\Gamma$ is faithful.

\medbreak

\pref{NotInvt-G} It can be shown that if $\mu_X$ is $G$-invariant, then $\mu_Z$ is $\Gamma$-invariant, which contradicts~\pref{NotInvt-Gamma}. Indeed, it is a general property of induced actions that $\mu_Z$ is $\Gamma$-invariant if and only if $\mu_X$ is $G$-invariant, but see \cref{NotInvtPf} for a proof of the direction we need, that is specialized to our situation.

\medbreak

\pref{NotInvt-=}
Suppose $\mu_X = \mu_X^P$. Since $\mu_X^P$ is $P$-invariant, this assumption implies that $\mu_X$ is $P$-invariant. However, $\mu_X$ is also known to be $K$-invariant \fullcsee{nuGDefn}{G}. These two subgroups generate all of~$G$, so we conclude that $\mu_X$ is $G$-invariant. This contradicts~\pref{NotInvt-G}.

\medbreak

\pref{NotInvt-P}
Suppose $\mu_X^P$ is $K$-invariant. This means $k_* \mu_X^P = \mu_X^P$ for all $k \in K$, so
	\[ \mu_X 
	= \int_K k_* \mu_X^P \, dm_K(k)
	= \int_K \mu_X^P \, dm_K(k)
	= \mu_X^P . \]
This contradicts~\pref{NotInvt-=}.
\end{proof}

\subsection{Contraction on \texorpdfstring{$\RR$}{R}-orbits and the Key Proposition}

\begin{defn}[{\cite[\S7.1, p.~43]{DeroinHurtado}}] \label{chiDefn}
We define $\chi_P \colon A \to \RR$ by
	\[ \chi_P(a) = \int_X \log \, D^\leaf_{\!a} (x) \, d \mu_X^P(x) . \]
(There is a technical issue here: we know that $D^\leaf_{\!a} (x)$ exists for $\mu_X$-a.e.\ $x \in X$, but the definition assumes that it exists for $\mu_X^P$-a.e.\ $x \in X$. See \cref{chiDefnIssue} for a proof that the derivative exists $\mu_X^P$-a.e.)
\end{defn}

\begin{lem}[see \cref{chiPHomo}, {\cite[\S7.1, p.~43]{DeroinHurtado}}] \label{chiPHomoOverview}
It follows from the Chain Rule that $\chi_P$ is a homomorphism.
\end{lem}

Notice that if $\chi_P(a) < 0$, then there is constant $C > 1$, such that if $x$ and~$y$ are two nearby points in the same $\RR$-orbit, then, on average, 
	\[ |ax - ay| < \frac{|x - y|}{C}  . \]
Applying~$a$ repeatedly should decrease the distance by an exponential factor, so one would expect the distance to tend to~$0$ in the limit. This is indeed true (if we start with two points that are not too far apart):

\begin{lem}[local contraction on $\RR$-orbits, see \cref{chiPLocal}, {\cite[Lem.~7.1]{DeroinHurtado}}] \label{chiPLocalOverview}
If $a \in A$, such that $\chi_P(a) < 0$, then for $\mu_X^P$-a.e.~$x \in X$, there exists $\epsilon = \epsilon(x) > 0$, such that
	\[ \text{$a^n (x + \epsilon) \ - \ a^n (x - \epsilon) \ \to \ 0$ \quad as $n \to \infty$} . \]
\end{lem}

With additional work, it can be shown that $a$ contracts entire $\RR$-orbits, not just $\epsilon$-intervals. 
The resulting contraction property will be the key to the proof of \cref{padicThm}. 

\begin{key}[global contraction on $\RR$-orbits, cf.\ {\cite[Prop.~7.7]{DeroinHurtado}}] \label{keyprop}
For all $a \in A$, if $\chi_P(a) < 0$, then 
	\[ \text{for $\mu_X^P$-a.e.\ $x \in X$, \ for all $y \in x + \RR$, \quad $a^n \, x - a^n \, y \ \tends \ 0$} . \]
\end{key}

\begin{proof}
See \cref{keypropPfShort} for an outline. More details can be found in \cref{globalSect}.
\end{proof}

\subsection{Invariance under the centralizer}

\begin{notation}
For $a \in A$, let
	\[ U_a^+ = \bigset{ u \in G }{ \begin{matrix} a^{n} u a^{-n} \to 1 \\[-2pt] \text{as $n \to -\infty$} \end{matrix}}
	\qquad \begin{pmatrix} \text{This is the \emph{expanding horospherical}} \\ \text{\emph{subgroup} corresponding to~$a$.} \end{pmatrix}
	 \]
\end{notation}

\Cref{keyprop} has the following consequence, which will be used to show (by a bootstrapping argument) that $\mu_X^P$ is $G$-invariant. This contradicts \fullcref{NotInvt}{P}.

\begin{cor}[cf.\ {\cite[Prop.~8.3]{DeroinHurtado}}] \label{C_G(a)}
If $a \in A$, such that $\chi_P(a) < 0$ and $U_a^+ \subseteq P$, then $\mu_X^P$ is $C_G(a)$-invariant. 
\end{cor}

\begin{proof}[Idea of proof]
Let $c \in C_G(A)$. We wish to show $c_* \mu_X^P = \mu_X^P$. 

The measure $\mu_X^P$ might not be ergodic for the action of~$a$, so we should work with ergodic components of the measure, but, for simplicity, we ignore this technical issue. (See \cref{C_G(a)Details} for a comment on this.)
Then, by the Pointwise Ergodic Theorem~\pref{PointwiseErgThm}, there exists $x \in X$, such that, for all $f \in C(X)$, we have
	\begin{align} \label{C_G(a)-Birkhoff}
	 \lim_{n \to \infty} \frac{1}{n} \sum_{i = 1}^n f(a^i \, x) = \int_X f \, d\mu_X^P 
	 . \end{align}
We say that $x$ is a ``Birkhoff-generic point'' for~$\mu_X^P$ (w.r.t.~$a$),
and we call the average on the left-hand side a ``Birkhoff average\rlap.''
Furthermore, since~\pref{C_G(a)-Birkhoff} holds for $\mu_X^P$-a.e.\ $x \in X$, we see from \cref{keyprop} that we may assume 
	\begin{align} \label{C_G(a)Pf-contract}
	\text{$a^n \, x \ - \ a^n \, y \quad \tends \quad 0$ \quad for all $y \in x + \RR$}
	 . \end{align}

Since $\mu_X^P$ is $P$-invariant, the measure $c_*\mu_X^P$ is ${}^c\!P$-invariant, where ${}^c\!P = c P c^{-1}$ (a conjugate of~$P$).
Define $\pi \colon X \to G/\Gamma$ by $\pi \bigl( [(h, z)] \bigr) = h \Gamma$, so $\pi$ is a well-defined $G$-equivariant map \fullcsee{Xbundle}{equi}.
Then $\pi_* (c_*\mu_X^P)$ is a ${}^c\!P$-invariant probability measure on $G/\Gamma$, so we see from \fullcref{PActsMinimal}{set} that its support is all of $G/\Gamma$. Since $\pi(cx)$ is Birkhoff generic for this measure (w.r.t.~$a$), we conclude that the $\langle a \rangle$-orbit of~$\pi(cx)$ is dense in $G/\Gamma$. Hence, we may choose $i \in \ZZ$, such that $\pi(a^i c x)$ is very close to~$\pi(x)$.

Now, if we let 
	\[ P_a^- = \bigset{ g \in G }{ \begin{matrix} \text{$a^{n} g a^{-n}$ is bounded} \\[-2pt] \text{for all $n \ge 0$} \end{matrix}} ,\]
then the Lie algebra~$\Lie g$ is the direct sum of the Lie algebra of~$P_a^-$ and the Lie algebra of~$U_a^+$. Hence, the conclusion of the preceding paragraph implies there exist small $g_- \in P_a^-$ and $u^+ \in U_a^+$, such that 
	\begin{align*} 
	\pi(a^i \, cx) = g^- u^+ \, \pi(x) = \pi(g^- u^+ x) 
	. \end{align*}
Writing $x = [(h, z)]$ with $h \in G$ and $z \in Z$, this means there exists $\gamma \in \Gamma$, such that
	\begin{align*} 
	a^i \, ch \gamma  = g^- u^+ h
	, \end{align*}
so (recalling that the $\RR$-orbits in~$Z$ are $\Gamma$-invariant) we have
	\begin{align} \label{C_G(a)Pf-SameLeaf}
	x
	&= c^{-1}a^{-i} a^i c x 
	= c^{-1}a^{-i} [(a^i c h, z)] 
	= c^{-1}a^{-i} [(a^i c h \gamma, \gamma^{-1} z)]
	\\&\notag
	= c^{-1}a^{-i} [(g^- u^+ h, \gamma^{-1} z)]
	 \in  c^{-1}a^{-i} [(g^- u^+ h, z)] + \RR
	= c^{-1}a^{-i} g^- u^+ x + \RR
	. \end{align}

Let $x_c = a^i c x$ and $x_0 = u^+ x$. 
We see from~\pref{C_G(a)Pf-SameLeaf} and~\pref{C_G(a)Pf-contract} that
	\[ a^n x - a^n (c^{-1}a^{-i} g^- x_0) \to 0 . \]
Then, since $a^i c$ commutes with~$a^n$, and is uniformly Lipschitz on the $\RR$-orbits \csee{LipOnLeaves}, we can multiply by $a^i c$ to conclude that
	\begin{align} \label{C_G(a)Pf-xcContract}
	a^n x_c - a^n(g^- x_0) \to 0 
	. \end{align}
However: 
	\begin{itemize}
	\item $d(a^n x_0, a^n x_c) \le d( a^n x_0, a^n g^- x_0) + d( a^n x_c, a^n g^- x_0)$.
	\item Since $g_- \in P^-_a$, we know that $d( a^n x_0, a^n g^- x_0)$ stays bounded at about the size of~$g^-$, so it is uniformly small.
	\item $d( a^n x_c, a^n g^- x_0) \le d_{\leaf}(a^n x_c, a^n g^- x_0) \to 0$ (by \pref{C_G(a)Pf-xcContract}).
	\end{itemize}
Therefore $a^n x_0$ is close to $a^n x_c$ for all~$n$, so 
	\[ \text{$x_0$ and $x_c$ have almost the same Birkhoff averages.} \]

Now, we have another technical issue. Since $x_0 = u^+ \, x$ and $x$ is Birkhoff generic for $\mu_X^P$, we would like to be able to say that $x_0$ is Birkhoff generic for $u^+_* \mu_X^P = \mu_X^P$. However, $u^+$ does not centralize~$a$ (far from it!), so this is not obvious. It is true for most values of~$u^+$ though: since $\mu_X^P$ is $U_a^+$-invariant, one can show (for a.e.~$x$) that the statement is true for $m_{U_a^+}$-a.e.\ choice of~$u^+$ in~$U_a^+$. By perturbing the generic point~$x$, we can change the value of~$u^+$ to get it into the good set (see \cite[p.~50]{DeroinHurtado}). Therefore, we can assume that $x_0$ is indeed Birkhoff generic for~$\mu_X^P$. 

On the other hand, $x_c$ is Birkhoff generic for $c_* \mu_X^P$. Combining this with the conclusions of the preceding two paragraphs, 
we conclude that $\mu_X^P$ is $\epsilon$-close to $c_* \mu_X^P$ (for every $\epsilon > 0$). By letting $\epsilon$ tend to~$0$, we conclude that $\mu_X^P = c_* \mu_X^P$. 
\end{proof}

\begin{rem}
The assumption in \cref{C_G(a)} that $U_a^+ \subseteq P$ simply means that $a$ is in a particular closed Weyl chamber in~$A$.
\end{rem}

\begin{note} \label{NotVacuous}
\Cref{C_G(a)} is not vacuous: there is an element $a$ of~$A$, such that $\chi_P(a) < 0$ and $U_a^+ \subseteq P$ \csee{ChiNontriv}.
\end{note}

\subsection{Completion of the proof by propagating invariance}

We can now complete the proof of \cref{padicThm}, by obtaining the following contradiction to \fullcref{NotInvt}{P}. (This is where we use the assumption that $\rank_{\QQ_p}G \ge 2$. If $G$ has rank one, then $C_G(a)$ is equal to~$A$, so \cref{C_G(a)} cannot show that $\mu_X^P$ is invariant under anything more than~$A$.) The contradiction implies that our assumption~\pref{AssumeActsOnR} that there is a nontrivial, orientation-preserving action of~$\Gamma$ on~$\RR$ must be false. So it does indeed complete the proof of \cref{padicThm}. 

\begin{cor}[``Propagating Invariance'' {\cite[Thm.~8.1 (and~\S9)]{DeroinHurtado}}] \label{propagating}
The measure $\mu_X^P$ is $G$-invariant. 
\end{cor} 

\begin{proof}[Idea of proof] 
For concreteness, let us assume $G = \SL(3,\QQ_p)$. (See \cref{GeneralPropagatingPf} for an explanation of the general case.) Let $\weyl_P$ be the closed Weyl chamber corresponding to~$P$, so $U^+_a \subseteq P$ for all $a \in \weyl_P$. We may assume $P$ is upper triangular, as suggested in \fullcref{G=SL3}{P}.
Let 
	\[ \halfspace = \{\, a \in A \mid \chi_P(a) < 0 \,\} . \] 
This is an open halfspace in~$A$. 

\refstepcounter{caseholder}

\begin{case} \label{propagatingPf-easy}
Assume $\chi_P(a) < 0$ for all nontrivial $a \in \weyl_P$.
\end{case}
Let $a_1$ and~$a_2$ be on the two walls of~$\weyl_P$, as pictured in \fullcref{propagate-fig}{easy}. 
\begin{figure}
\[ \begin{matrix}
\includegraphics{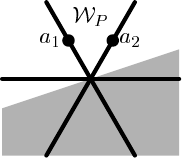} 
%
%
%
%
%
%
%
%
%
& \hskip 5mm &
\includegraphics{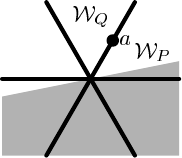} 
%
%
%
%
%
%
%
& \hskip 5mm &
\includegraphics{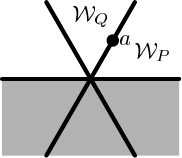} 
%
%
%
%
%
%
%
\\ \hbox{(a)} && \hbox{(b)} && \hbox{(c)} 
\end{matrix}
\]
\caption{The halfspace $\halfspace$ is unshaded.}
\label{propagate-fig}
\refstepcounter{subfigure}
\label{propagate-fig-easy}
\refstepcounter{subfigure}
\label{propagate-fig-interior}
\refstepcounter{subfigure}
\label{propagate-fig-bdry}
\end{figure}
Specifically, we may let
	\[ a_1 = \left[ \begin{matrix} 1/p &  &  \\  & 1/p &  \\  &  & p^2 \end{matrix} \right]
	\qquad \text{and} \qquad
	a_2 = \left[\begin{matrix} 1/p^2 \\ & p \\ && p \end{matrix}\right]
	. \]
By the assumption of this \lcnamecref{propagatingPf-easy}, we have $\chi_P(a_1) < 0$ and $\chi_P(a_2) < 0$. Also, we know $U^+_{a_1} \subseteq P$ and $U^+_{a_2} \subseteq P$, because $a_1, a_2 \in \weyl_P$. Indeed,
	\[ U^+_{a_1} = \left[ \begin{smallmatrix}  1 & & * \\  & 1 & *  \\ && 1 \end{smallmatrix}\right] \subset P
	\qquad \text{and} \qquad
	U^+_{a_1} =  \left[\begin{smallmatrix} 1 & * & * \\ & 1 &  \\ && 1 \end{smallmatrix}\right] \subset P . \]
So we see from \cref{C_G(a)} that $\mu_X^P$ is invariant under
	\[ C_G(a_1) =  \left[ \begin{smallmatrix} * &*\\ * & * \\ && 1 \end{smallmatrix}\right]
	\qquad \text{and} \qquad
	C_G(a_2) =  \left[\begin{smallmatrix} 1 & \\ & * & * \\ & *& * \end{smallmatrix}\right]
	. \]
These two centralizers generate~$G$, so we conclude that $\mu_X^P$ is invariant under all of~$G$. 

\begin{case}
The remaining case.
\end{case}
We know that:
	\begin{itemize}
	\item $\halfspace$ contains some element of~$\weyl_P$ (by \cref{NotVacuous}), 
	and
	\item $\halfspace$ does not contain every nontrivial element of~$\weyl_P$ (by the assumption of this case).
	\end{itemize}
So either the boundary of~$\halfspace$ passes through the interior of~$\weyl_P$ (as in \fullcref{propagate-fig}{interior}), 
or the boundary of~$\halfspace$  contains one the boundary rays of~$\weyl_P$, and the rest of~$\weyl_P$ is in~$\halfspace$ (as in \fullcref{propagate-fig}{bdry}). 
In either case, one of the two boundary rays of~$\weyl_P$ is in~$\halfspace$.

Let $\weyl_Q$ be the Weyl chamber on the other side of this boundary ray, with corresponding minimal parabolic subgroup~$Q$, and let $a$ be a nontrivial point on the ray. Then $a \in \halfspace$, which means $\chi_P(a) < 0$. (Also, $a \in \weyl_P$, so $U_a^+ \subseteq P$.) So we see from \cref{C_G(a)} that $\mu_X^P$ is $C_G(a)$-invariant. However, we also know by the definition of~$\mu_X^P$ that it is $P$-invariant. So $\mu_X^P$ is invariant under the subgroup $\langle C_G(a), P \rangle$. This is a parabolic subgroup of~$G$ (because every subgroup that contains a parabolic subgroup is also a parabolic subgroup). In fact, since $a$ is on the boundary between $\weyl_P$ and~$\weyl_Q$, this subgroup is equal to $\langle P, Q \rangle$, and therefore contains~$Q$. So $\mu_X^P$ is $Q$-invariant.

By the definition of~$\mu_X^P$, we also have $\mu_X = \int_K k_* \mu_X^P \, dm_K(k)$ \csee{FurstenburgMuP}. Hence, the uniqueness in \cref{FurstenburgMuP} implies that $\mu_X^Q = \mu_X^P$. Then we see from \cref{chiDefn} that $\chi_Q = \chi_P$. Since it is clear from \cref{propagate-fig}(\ref{propagate-fig-interior},\ref{propagate-fig-bdry}) 
that every nontrivial element of~$\weyl_Q$ is in~$\halfspace$, we have $\chi_P(a) < 0$ for all nontrivial $a \in \weyl_Q$. Then, since $\chi_Q = \chi_P$, we conclude that the same inequality holds with $\chi_Q$ in the place of~$\chi_P$. Hence, \cref{propagatingPf-easy} applies with the parabolic subgroup~$Q$ in the place of~$P$.
\end{proof}

\subsection{Proof of \texorpdfstring{\cref{keyprop}}{the Key Proposition}} \label{keypropPfShort}
Let $a \in A$, such that $\chi_P(a) < 0$. We know from \cref{chiPLocalOverview} that the action of~$a$ contracts small intervals in $\RR$-orbits (a.e.). In order to establish \cref{keyprop}, we need to show that the contraction happens on the entire $\RR$-orbit. To this end, we will employ contraction by generic elements of~$G$, rather than only by elements of~$A$.

For $g \in G$, we have $\mu_X(gX) = \mu_X(X) = 1$, so (by \cref{RadonNikodym}) $\int_X D^\leaf_g \, d\mu_X = 1$. As $\mu_X$ is not $G$-invariant (see \fullcref{NotInvt}{G}), we also know that $D^\leaf_g$ is not identically equal to~$1$ for all~$g$, so, since $\log$ is a concave function, we conclude from Jensen's Inequality that 
	\begin{align} \label{Liap}
	 \int_G \int_X \log D^\leaf_g(x) \, d\mu_X(x) \,d\mu_G(g) < 0 
	 . \end{align}
This implies that (on average) applying a random element of~$G$ contracts distances that are small. Much like what we saw in \cref{chiPLocalOverview} for~$a^n$, this implies that repeated application of random elements of~$G$ will contract small intervals in the $\RR$-orbits. However, we also have the following global contraction property of the $\Gamma$-action on each $\RR$-orbit:

\begin{prop}[see \cref{EveryActContracts}, {\cite[Prop.~6.3]{DeroinHurtado}}] \label{EveryActContractsOverview}
For every action of\/~$\Gamma$ on~$\RR$ without fixed points, there is a compact subset~$I$ of\/~$\RR$, such that for every $\epsilon > 0$ and all $s,t \in \RR$, there exists $\gamma \in \Gamma$, such that
	\[ \text{$|\gamma s - \gamma t| < \epsilon$ \ and \ $\{\gamma s, \gamma t\} \subset I$}  . \]
\end{prop}

By using this, it is possible to prove that repeated application of random elements of~$G$ yields contraction on entire $\RR$-orbits, rather than only locally, and that the contraction is exponentially fast:

\begin{prop}[see \cref{RandomContract}, {\cite[Props.~6.4 and~6.8]{DeroinHurtado}}] \label{RandomContractOverview}
There is a constant $\delta < 1$, such that, for $\mu_X$-a.e.~$x \in X$, for all $y \in \RR x$, and for $\mu_G^\infty$-a.e.\ sequence $(g_i)_{i=1}^\infty$ of elements of~$G$, we have
	\[ \text{$|g_n g_{n-1} \cdots g_1 x \ - \ g_n g_{n-1} \cdots g_1 y| < \delta^n$ \quad for all sufficiently large~$n$} . \]
\end{prop}

We have been proving the $p$-adic case in this paper, but the next part of the argument is easier to explain for real groups, so we will switch to that setting for the remainder of this \lcnamecref{keypropPfShort}:
	\[ \text{\emph{Assume, contrary to \cref{pGGamma}, that $G$ is a connected real semisimple Lie group.}} \]

The group $G$ acts by isometries on its associated symmetric space~$K \backslash G$, which is a Riemannian manifold of nonpositive sectional curvature. (The $p$-adic case uses a Bruhat-Tits building, instead of a symmetric space.) Applying a random walk in~$G$ to a point in $K \backslash G$ yields a sequence of points, and it can be shown that the negative curvature of $K \backslash G$ implies that this sequence almost surely moves along a geodesic, up to a sublinear error:

\begin{thm}[Karlsson-Margulis {\cite[Thm.~2.1]{KarlssonMargulis}}] \label{tracking}
For $\mu_G^\infty$-a.e.\ sequence $(g_i)_{i=1}^\infty$ of elements of~$G$, there is a geodesic $\mathsf{a}_t$ in~$K \backslash G$, with $\mathsf{a}_0 = K$, such that
	\[ \frac{d  \bigl( K g_n g_{n-1} \cdots g_1, \ \mathsf{a}_n \bigr)}{n} \tends 0 . \]  
\end{thm}

\begin{defn}
Fix a nonempty, open subset~$\open$ of~$G$. Then $\open$ generates~$G$, so we can define the corresponding \emph{word length}: for $g \in G$, we let
	\[ \ell_O(g) = \min\{\, r \in \NN \mid g \in \open^r \,\} , \]
where $\open^r = \{\, h_1 h_2 \cdots h_r \mid \text{$h_i \in \open$ or $h_i^{-1} \in \open$} \,\}$.
\end{defn}

Geodesics in $K \backslash G$ come from one-parameter subgroups of ($K$-conjugates of)~$A$, and every element of~$A$ is $K$-conjugate to an element of~$\weyl_P$, so \cref{tracking} can be restated as follows:

\begin{cor}[{\cite[Thm.~3.18]{DeroinHurtado}}] \label{trackingA}
For $\mu_G^\infty$-a.e.\ sequence $(g_i)_{i=1}^\infty$ of elements of~$G$, there exist $k \in K$ and $\widehat {a_P} \in \weyl_P \smallsetminus \{\1\}$, such that 
	\[ \frac{\ell_O ( g_n g_{n-1} \cdots g_1 \, k \, \widehat{a_P}^{-n})}{n}  \tends 0 . \]
Furthermore, as $(g_i)_{i=1}^\infty$ varies, the element~$k$ is uniformly distributed in~$K$ \textup(with respect to the Haar measure~$m_K$\textup).
\end{cor}

We also have the following important technical result:

\begin{prop}[Benoist-Quint {\cite[Thm.~10.9(b), p.~158]{BenoistQuintBook}}]
The element~$\widehat{a_P}$ is a well-defined element of~$\weyl_P$ that is independent of $(g_i)_{i=1}^\infty$.
\end{prop}

Combining these results has the following consequence:

\begin{prop}[{\cite[Prop.~7.6]{DeroinHurtado}}] \label{ExistsContractingOverview}
For $\mu_G^\infty$-a.e.\ sequence $(g_i)_{i=1}^\infty$ of elements of~$G$, $\mu_X^P$-a.e.\ $x \in X$ and every $y \in x + \RR$, we have
	\[ \widehat{a_P}^n \, x \  - \ \widehat{a_P}^n \, y \  \tends \ 0 . \]
\end{prop},

\begin{proof}[Idea of proof]
We will ignore some technical issues, including the difference between~$\mu_X$ and~$\mu_X^P$.
 (\Cref{ExistsContractingDetails} explains how to get from $\mu_X$ to~$\mu_X^P$, by using the fact that $U^+_a$ is contained in~$P$, and therefore preserves~$\mu_X^P$.)
 
Since $\mu_X$ is $K$-invariant, we see that, for $m_K$-a.e.\ $k \in K$, the conclusion of \cref{RandomContractOverview} holds with $kx$ and~$ky$ in the place of $x$ and~$y$. Then the uniform distribution in the conclusion of \cref{trackingA} implies that we can choose $k$ and~$(g_i)_{i=1}^\infty$ so that the conclusions of \cref{RandomContractOverview} and \cref{trackingA} hold simultaneously. That is, 
	\[ \text{$| g_n g_{n-1} \cdots g_1 k x \ - \  g_n g_{n-1} \cdots g_1 k y| \tends 0$ exponentially fast,}\]
 and the difference between $g_n g_{n-1} \cdots g_1 \, k$ and $\widehat{a_P}^n$ is the product of a sublinear number of elements of~$\open$. Since elements of~$\open$ are uniformly Lipschitz, we conclude that $| \widehat{a_P}^n x \ - \  \widehat{a_P}^n y|$ also tends to~$0$ (exponentially fast). 
 \end{proof}

We have established that the single element~$\widehat{a_P}$ has global contraction. We can use this to prove that the elements in a halfspace of~$A$ also contract entire $\RR$-orbits.
In other words, we now prove \cref{keyprop}.

\begin{proof}[\bf Proof of \cref{keyprop}]
Fix $a \in A$. For each $x \in X$, 
let
	\[ f(x) = \sup \{\, t \in \RR^{\ge 0} \mid  a^n (x+t) - a^n x \ \tends \ 0 \,\} . \]
For any $n \in \ZZ^+$, the element $\widehat{a_p}^n$ commutes with~$a$ (and is uniformly Lipschitz on the $\RR$-orbits by \cref{LipOnLeaves}), so we have
	\begin{align*}
	 \widehat{a_p}^n \bigl( x + f(x) \bigr) = \widehat{a_p}^n x + f(\widehat{a_p}^n \, x)  
	  .\end{align*}
If $f(x) < \infty$, this implies
	\begin{align*}
	f(\widehat{a_p}^n \, x) \quad = \quad  \widehat{a_p}^n \bigl( x + f(x) \bigr) \ - \ \widehat{a_p}^n x \quad \tends \quad 0 
	&& \text{(by \cref{ExistsContractingOverview})}
	. 
	\end{align*}
We can now conclude from  the Pointwise Ergodic Theorem~\pref{PointwiseErgThm} that $\int_X f \, d\mu_X^P = 0$, so $f(x) = 0$ $\mu_X^P$-a.e. However, if $\chi_P(a) < 0$, then we see from \cref{chiPLocalOverview} that $f(x) > 0$ for $\mu_X^P$-a.e.~$x$. This is a contradiction, and completes the proof of \cref{keyprop}. 
\end{proof}

\section{Some additional explanations} \label{DeroinHurtadoPfSect}

For the interested reader, we provide some arguments that were omitted from \cref{OverviewSect}.

\subsection{Construction of the almost-periodic space~\texorpdfstring{$Z$}{Z}}

\begin{thm}[proof of \cref{AlmPerSpace}, {\cite[Thm.~1.1]{Deroin-AlmPer}, \cite[Thm.~5.4]{DeroinHurtado}}] \label{AlmPerSpacePf}
There is a \textup(nonempty\textup) compact metrizable space $Z$, such that: 
	\begin{enumerate}
	\item \label{AlmPerSpacePf-act}
	$\Gamma$ acts on~$Z$ \textup(by homeomorphisms\textup) with no global fixed point, 
	\item \label{AlmPerSpacePf-R}
	$\RR$ has a continuous action on $Z$ that is free, 
	\item \label{AlmPerSpacePf-orbits}
	each\/ $\RR$-orbit is\/ $\Gamma$-invariant, 
	and 
	\item \label{AlmPerSpacePf-orpres}
	the action of\/~$\Gamma$ on each\/ $\RR$-orbit is orientation preserving.
	\end{enumerate}
	Furthermore, two additional nice properties \pref{GammaLipschitz} and \pref{MeanDisplacement} hold.
\end{thm}

\begin{proof}[Sketch of proof]
Fix a finite, symmetric subset~$S$ that generates~$\Gamma$ and contains~$\1$. 
A general action of~$\Gamma$ on~$\RR$ may be very badly behaved, but it has been shown \cite[Props.~8.1 and~8.4]{DeroinEtAl-Symmetric} that if an action exists, then there is an action $\varphi \colon \Gamma \to \Homeo^+(\RR)$, and a constant~$C$, such that for all $s,t \in \RR$, and all $\gamma \in S$:
	\begin{itemize}
	\item (Lipschitz) \quad
		$ |\varphi_\gamma(s) - \varphi_\gamma(t) | \le C |s - t|$,
		and
	\item (bounded displacement) \quad $ |\varphi_\gamma(t) - t | \le C$.
	\end{itemize}
Thus, we can fix a constant~$C$ such that the following set is nonempty:
	\[ Z_0 \coloneqq  \bigset{ 
		\begin{matrix} 
		\varphi \colon \Gamma \to \Homeo^+(\RR) \\
		\text{(homomorphism)}
		\end{matrix}
		}{
		\begin{matrix}
		\text{for all $\gamma \in S$ and all $s,t \in \RR$,} \\
		\text{we have \ }
		|\varphi_\gamma(s) - \varphi_\gamma(t) | \le C |s - t| \\
		\text{and \ }  |\varphi_\gamma(t) - t | \le C
		\end{matrix}
		}. \]
The trivial action is one of the elements of~$Z_0$. To eliminate this, we add a nontriviality condition:
	\begin{align} \label{ZNontriv}
	Z = \bigset{ \varphi \in Z_0 }{ \forall t \in \RR, \ \int_\Gamma \bigl( \varphi_\gamma(t) - t \bigr)^2 \, d \mu_\Gamma(\gamma) > \frac{1}{C} }
	. \end{align}
If $1/C$ is small enough, then $Z$ will be nonempty.

For each $\gamma \in \Gamma$, the bounded displacement condition implies that $\{\, \varphi_\gamma \mid \varphi \in Z \,\}$ is uniformly bounded on any compact interval, and the Lipschitz condition implies that this set of functions is equicontinuous. Therefore, we see from the Arzelà–Ascoli Theorem that $Z$ is compact.

This completes the construction of the space~$Z$. We now verify that it has the desired properties~(\ref{AlmPerSpacePf-act}--\ref{AlmPerSpacePf-orpres}). 

\medbreak

\pref{AlmPerSpacePf-R} The group~$\RR$ acts on~$Z$ by via conjugation by translations: for $t \in \RR$, let
	\[ (t * \varphi)_\gamma (s) = \varphi_\gamma(s + t) - t . \]

Suppose this $\RR$-action is not free, so there is some nonzero $t \in \RR$ and some $\varphi \in Z$, such that $t * \varphi = \varphi$. This means that for all $s \in \RR$ and all $\gamma \in \Gamma$, we have
	\[ \varphi_\gamma (s + t) - t = \varphi_\gamma(s) , \]
so the action of~$\Gamma$ centralizes the translation by~$t$. Hence, the action of~$\Gamma$ on~$\RR$ factors through to a well-defined action on the quotient $\RR / t\ZZ$. Then \fullcref{SSLattFacts}{CircleAction} tells us that this action must have a fixed point $s + t\ZZ$ (after replacing $\Gamma$ with a finite-index subgroup). So the subset $s + t\ZZ$ of~$\RR$ is $\Gamma$-invariant. The group~$\Gamma$ must act on this set by order-preserving permutations, which means that it acts by translations. But the group of translations is abelian, and the abelianization of~$\Gamma$ is finite \fullcsee{SSLattFacts}{abel}. So we conclude that $\Gamma$ fixes~$s$ (and also fixes all of the other points in $s + t \ZZ$. More precisely, $\varphi_\gamma(s) =s$ for every $\gamma \in \Gamma$. This contradicts the nontriviality condition in~\pref{ZNontriv}.

\medbreak

\pref{AlmPerSpacePf-act} We define an action of~$\Gamma$ on~$Z$ by letting an element~$\alpha$ of~$\Gamma$ act on the point~$\varphi$ of~$Z$ via the action of the element $\varphi_\alpha(0)$ of~$\RR$, so ${}^\alpha \! \varphi = \varphi_\alpha(0) * \gamma$. More concretely:
	\[ {}^\alpha \! \varphi_\gamma(t) = \bigl( \varphi_\alpha(0) * \varphi)_\gamma(t) =  \varphi_\gamma \bigl( t + \varphi_\alpha(0) \bigr) - \varphi_\alpha(0) . \]

It may not be obvious that this defines an action, so we provide a verification. Noting that, for $\alpha \in \Gamma$ and $t \in \RR$, we have
	\begin{align} \label{gamma(z+t)}
	 {}^\alpha (t * \varphi)
	= (t * \varphi)_\alpha(0) * (t * \varphi)
	= [ \varphi_\alpha(0 + t) - t] * (t * \varphi)
	= \varphi_\alpha(t) * \varphi
	, \end{align}
we see that
	\[{}^\alpha \bigl( {}^\beta \! \varphi \bigr)
	= {}^\alpha( \varphi_\beta(0) * \varphi)
	= \varphi_\alpha \bigl( \varphi_\beta(0) \bigr) * \varphi
	= \varphi_{\alpha \beta}(0) * \varphi
	= {}^{\alpha\beta} \varphi
	. \]

Suppose this $\Gamma$-action has a fixed point: ${}^\alpha \varphi = \varphi$ for every $\alpha \in \Gamma$. This means $\varphi_\alpha(0) * \varphi = \varphi$, so $\varphi_\alpha(0) = 0$ (because the $\RR$-action is free). Since this holds for every $\alpha \in \Gamma$, we conclude that $0$ is a fixed point of the action given by~$\varphi$. This contradicts the nontriviality condition in~\pref{ZNontriv}.

\medbreak

\pref{AlmPerSpacePf-orbits} Since elements of~$\Gamma$ act on each point via an element of~$\RR$, it is clear that the each $\RR$-orbit is $\Gamma$-invariant. 

\medbreak

\pref{AlmPerSpacePf-orpres}
If $s < t$, then $\varphi_\alpha(s) < \varphi_\alpha(t)$ (because $\varphi_\alpha$ is orientation preserving). So we see from~\pref{gamma(z+t)} that the action of~$\Gamma$ on each $\RR$-orbit is orientation preserving.
\end{proof} 

\begin{lem} \label{GammaLipschitz}
The action of\/ $\Gamma$ on each\/ $\RR$-orbit is Lipschitz. More precisely, for all $\alpha \in S$, all $t \in \RR$, and all $\varphi \in Z$, there is some $s \in \RR$, such that
	\[ {}^\alpha (t * \varphi) = s * {}^\alpha \varphi  
	\qquad \text{and\/ $|s| \le C |t|$} . \]
\end{lem}

\begin{proof}
By~\pref{gamma(z+t)}, we may take $s = \varphi_\alpha(t)$, and then we see from the bounded displacement property in the definition of~$Z_0$ that 
	\[ |s| = |\varphi_\alpha(t)| \le |\varphi_\alpha(t) - t| + |t| \le C|t| + |t| = (C + 1) |t| . \qedhere \]
\end{proof}

\subsection{The measures \texorpdfstring{$\mu_Z$ and~$\mu_X$}{on Z and X} are stationary}

It is an immediate consequence of work of B.\,Deroin, V.\,Kleptsyn, A.\,Navas, and K.\,Parwani \cite[Prop.~8.1]{DeroinEtAl-Symmetric} that left-orderable groups have an action on~$\RR$ in which the mean displacement of every point is~$0$. Therefore, the almost-periodic space~$Z$ can be constructed to have this property on every $\RR$-orbit:
	 \begin{align} \label{MeanDisplacement}
	  \forall z \in Z, \quad \int_\Gamma (\gamma z - z) \, d  \mu_\Gamma (\gamma) = 0
	  . \end{align}

\begin{rem}
It is shown in \cite[Prop.~8.4]{DeroinEtAl-Symmetric} that any action that satisfies~\pref{MeanDisplacement} is Lipschitz and has bounded displacement. Indeed, the existence of an action with the latter two properties (which was fundamental in the construction of the almost-periodic space~$Z$) is obtained as a consequence of the existence of an action satisfying~\pref{MeanDisplacement}.
\end{rem}

From~\pref{MeanDisplacement}, we obtain the following fact. 

\begin{prop}[proof of \cref{nuGammaStationaryStated}, {\cite[Lem.~5.6]{DeroinHurtado}}] \label{nuGammaStationary}
The measure $\mu_Z$ is $\mu_\Gamma$-stationary. 
\end{prop}

\begin{proof}
For a.e.~$z \in Z$, we have
	\begin{align*}
	\int_\Gamma \frac{d \gamma_*\mu_Z}{d \mu_Z}(z) \, d  \mu_\Gamma (\gamma) 
	&= \int_\Gamma \frac{d \gamma^{-1}_*\mu_Z}{d \mu_Z}(z) \, d  \mu_\Gamma (\gamma) 
	\\&= \int_\Gamma \dleaf_\gamma(z)  \, d  \mu_\Gamma (\gamma) 
	\\&= \int_\Gamma \lim_{t \to 0} \frac{\gamma(z+ t) - z}{t} \, d  \mu_\Gamma (\gamma) 
	\\&= 1 + \int_\Gamma \lim_{t \to 0} \frac{\gamma(z+ t) - (z + t)}{t} \, d  \mu_\Gamma (\gamma) 
	\\&= 1 + \lim_{t \to 0} \frac{\int_\Gamma \bigl( \gamma(z+ t) - (z + t)   \bigr)\, d  \mu_\Gamma (\gamma) }{t} 
		&& \begin{pmatrix} \text{the integral is actually a}\\ \text{finite sum \csee{nuGammaDefn}} \end{pmatrix}
	\\&= 1 + \lim_{t \to 0} \frac{0}{t} 
		&& \text{\pref{MeanDisplacement}}
	\\& = 1
	. \end{align*}
So
	\[ \int_\Gamma \gamma_* \mu_Z \, d  \mu_\Gamma (\gamma) 
	=  \int_\Gamma \frac{d \gamma_*\mu_Z}{d \mu_Z} \, \mu_Z \, d  \mu_\Gamma (\gamma) 
	=  \left( \int_\Gamma \frac{d \gamma_*\mu_Z}{d \mu_Z} \, d  \mu_\Gamma (\gamma) \right) \cdot \mu_Z
	= 1 \cdot \mu_Z
	= \mu_Z . \qedhere \]
\end{proof}

\begin{cor}[proof of \cref{muXStationaryStated}, {\cite[Prop.~5.11]{DeroinHurtado}}] \label{muXStationary}
$\mu_X$ is $\mu_G$-stationary.
\end{cor}

\begin{proof}
Let $[(k_0,z_0)] \in X$, with $k_0 \in K$ and $z_0 \in Z$. 
Then
	\begin{align*}
	\int_G  \frac{d g_*\mu_X }{ d \mu_X} & \bigl( [(k_0,z_0)] \bigr) \, d \mu_G(g)
	\\&= \int_G  \frac{d g^{-1}_*\mu_X }{ d \mu_X} \bigl( [(k_0,z_0)] \bigr) \, d \mu_G(g)
	\\&= \int_G \dleaf_{g} \bigl( [(k_0,z_0)] \bigr) \, d\mu_G(g)
	\\&= \int_G \dleaf_{gk_0} \bigl( [(\1,z_0)] \bigr)  \cdot  \dleaf_{k_0^{-1}}\bigl( [(k_0,z_0)] \bigr) \, d\mu_G(g)
		&& \text{(Chain Rule)}
	\\&= \int_G \dleaf_{gk_0} \bigl( [(\1,z_0)] \bigr) \cdot 1 \, d\mu_G(g)
	\\&= \int_G \dleaf_{g} \bigl( [(\1,z_0)] \bigr) \, d\mu_G(g)
		&& \text{($\mu_G$ is right $K$-invariant (\fullref{nuGDefn}{G-Kinvt})}
	\\&= \int_\Gamma  \int_K \dleaf_{k \gamma} \bigl( [(\1,z_0)] \bigr) \, d m_K(k) \, d\mu_\Gamma(\gamma)
		&& \text{(definition of $\mu_\Gamma$)}
	\\&= \int_\Gamma  \int_K \dleaf_{\gamma} \bigl( [(\1,z_0)] \bigr) \, d m_K(k) \, d\mu_\Gamma(\gamma)
		&& \text{(Chain Rule and $\dleaf_k = 1$)}
	\\&= \int_\Gamma \dleaf_{\gamma} \bigl( [(\1,z_0)] \bigr) \, d\mu_\Gamma(\gamma)
		&& \text{(integral of a constant)}
	\\&= \int_\Gamma \dleaf_{\gamma} (z_0) \, d\mu_\Gamma(\gamma)
		&& \text{($\gamma \bigl( [(\1,z)] \bigr) = [(\gamma ,z)] =  [(\1,\gamma z)]$)}
	\\&= 1
		&& \text{(\cref{nuGammaStationary})}
	. \qedhere \end{align*}
\end{proof}

\begin{warn}
\Cref{muXStationary} is not a trivial consequence of \cref{nuGammaStationary}: a stationary measure for an action of a lattice subgroup does not usually lead in an obvious way to a stationary measure for the induced action of the ambient group. However,  D.\,Creutz \cite[\S3]{Creutz-stationary} has found a general construction.
\end{warn}

\begin{lem}[proof of \fullcref{NotInvt}{G}] \label{NotInvtPf}
If $\mu_X$ is $G$-invariant, then $\mu_Z$ is $\Gamma$-invariant.
\end{lem}

\begin{proof}
Given $\gamma \in \Gamma$ and $Z_0 \subseteq Z$, we wish to show $\mu_Z(Z_0) = \mu_Z(\gamma Z_0)$. Since $K$ is open, there is an open neighbourhood $K_0$ of~$\1$ in~$K$, such that $\gamma K_0 \gamma^{-1} \subseteq K$. Then
	\[ \gamma [K_0 \times Z_0] 
	= [\gamma K_0 \times Z_0] 
	= [\gamma K_0\gamma^{-1} \times \gamma Z_0]
	, \]
so
	\begin{align*}
	m_K(K_0) \cdot \mu_Z(Z_0)
	&= \mu_X \bigl( [K_0 \times Z_0] \bigr) 
		&& \text{(definition of~$\mu_X$)}
	\\&= \mu_X \bigl( \gamma [K_0 \times Z_0] \bigr)
		&& \text{(we are assuming that $\mu_X$ is $G$-invariant)}
	\\&= \mu_X \bigl( [\gamma K_0\gamma^{-1} \times \gamma Z_0] \bigr)
		&& \text{(definition of equivalence relation on $G \times Z$)}
	\\&= m_K (\gamma K_0 \gamma^{-1}) \cdot \mu_Z(\gamma Z_0)
		&& \text{(definition of~$\mu_X$)}
	. \end{align*}
However, 
	\[ m_K (\gamma K_0 \gamma^{-1})
	= m_G(\gamma K_0 \gamma^{-1})
	= m_G(K_0)
	= m_K(K_0) . \]
(The middle equality is because the semisimple group~$G$ is unimodular.)
Therefore, by cancelling a factor of $m_K (K_0)$, we conclude that $\mu_Z(Z_0) = \mu_Z(\gamma Z_0)$.
\end{proof}

\subsection{Contraction on the \texorpdfstring{$\RR$}{R}-orbits}

\begin{lem}[resolution of the technical issue in \cref{chiDefn}, cf.\ {\cite[Lem.~5.13]{DeroinHurtado}}] \label{chiDefnIssue}
$D^\leaf_{\!a} (x)$ exists for $\mu_X^P$-a.e.\ $x \in X$.
\end{lem}

\begin{proof}
Let 
	\[ \text{$B = \{\, x\in X \mid \text{$D^\leaf_{\!a} (x)$ does not exist} \,\}$ \quad (the ``bad'' set)}. \]
Suppose $[(k,z)] \in B$, and write $ak = k' \gamma$ with $k' \in K$ and $\gamma \in \Gamma$. Then $a$ acts on the entire $\RR$-orbit $[(k,z +\RR)]$ via the element~$\gamma$ of~$\Gamma$ \fullcsee{Xbundle}{OnLeaf}. Since $[(k,z)] \in B$, this implies that $\dleaf_\gamma(z)$ does not exist.
However, $K \gamma$ is open, so there is a neighbourhood~$\open$ of~$\1$ in~$K$, such that $a\open k \subseteq K \gamma$. Then $a$ acts via~$\gamma$ on every $\RR$-orbit in $[\open k \times Z]$. This implies $\open [(k,z)]  \subseteq B$.

For each $x \in B$, we have shown there is a neighbourhood $\open$ of~$\1$ in~$K$, such that $\open \, x \subseteq B$. If $\mu_X^P(B) > 0$, then there are a subset~$B'$ of positive measure, and a (small) neighbourhood~$\open'$ of~$\1$ in~$K$, such that $\open' B' \subset B$. Then, since $\mu_X = m_K \times \mu_Z$, we have
	\[ \mu_X(B) \ge \mu_X(\open' B') \ge m_K(\open') \cdot \mu_Z \bigl \{\, z \in Z \mid \bigl(K \times \{z\} \bigr) \cap B' \neq \emptyset \,\} \neq 0. \]
This contradicts the fact that $D^\leaf_{\!a} (x)$ exists for $\mu_X$-a.e.\ $x \in X$.
\end{proof}

\begin{lem}[proof of \cref{chiPHomoOverview}, {\cite[\S7.1, p.~43]{DeroinHurtado}}] \label{chiPHomo}
$\chi_P$ is a homomorphism.
\end{lem}

\begin{proof}
We have
	\begin{align*}
	 \chi_P(a_1 a_2)
	&= \int_X \log D^\leaf_{\!a_1 a_2} (x)  \, d \mu_X^P(x)
	\\&= \int_X \log D^\leaf_{a_1} (a_2 x)  \, d \mu_X^P(x) + \int_X \log D^\leaf_{\!a_2} (x)  \, d \mu_X^P(x)
		&& \text{(Chain Rule)}
	\\&= \int_X \log D^\leaf_{a_1} (a_2 x)  \, d \mu_X^P(x) + \chi_P(a_2) 
	. \end{align*}
Since $A \subseteq P$, we know that $\mu_X^P$ is $a_2$-invariant, so 
	\[  \int_X \log D^\leaf_{\!a_1} (a_2 x)  \, d \mu_X^P(x)
	= \int_X \log D^\leaf_{\!a_1} (x)  \, d \mu_X^P(x)
	= \chi_P(a_1) . \qedhere \]
\end{proof}

Since the measure $\mu_Z$ is ergodic for the $\RR$-action, it can be shown that the measure $\mu_X^P$ is ergodic for the action of~$P$:

\begin{lem}[{\cite[Lem.~6.5]{DeroinHurtado}}] \label{muXPErgodic}
The $P$-invariant measure $\mu_X^P$ is ergodic.
\end{lem}

\begin{proof}[Idea of proof]
Suppose $\mu_X^P$ is not ergodic, so $\mu_X^P = s \mu_1 + t \mu_2$ is a convex combination of two different $P$-invariant probability measures (with $st \neq 0$). Then
	\[ \mu_X 
	= \int_K k_* \mu_X^P \, dm_K
	=  s\int_K k_* \mu_1 \, dm_K +  t \int_K k_* \mu_2 \, dm_K
	\]
is a convex combination of two $G$-stationary measures. Also, the uniqueness in \cref{FurstenburgMuP} implies that neither of these stationary measures is equal to~$\mu_X$. So $\mu_X$ is not ergodic as a stationary measure.

Hence, there is a $G$-invariant function~$f$ that is \ul{not} constant (a.e.). 
The Pointwise Ergodic Theorem~\pref{PointwiseErgThm} has an analogue for random walks with a stationary measure \cite[Cor.~4.8]{DeroinHurtado} (cf.\ the proof of \cref{ChiNontriv}). 
Combining this with \cref{RandomContract} (contraction on $\RR$-orbits) implies that $f$ is essentially constant on a.e.\ $\RR$-orbit in~$X$. Since the $\RR$-action on~$Z$ is ergodic, this implies that $f \bigl( [(h, z)] \bigr)$ is essentially independent of~$z$. But the $G$-invariance implies that it is also essentially independent of~$h$. So $f$ is constant~(a.e.). This is a contradiction.
\end{proof}

\begin{lem}[proof of \cref{chiPLocalOverview}, {\cite[Lem.~7.1]{DeroinHurtado}}] \label{chiPLocal}
If $a \in A$, such that $\chi_P(a) < 0$, then for $\mu_X^P$-a.e.~$x \in X$, there exists $\epsilon = \epsilon(x) > 0$, such that
	\[ \text{$a^n (x - \epsilon) \ - \ a^n (x + \epsilon) \ \to \ 0$ \quad as $n \to \infty$} . \]
\end{lem}

\begin{proof}
For (small) $\eta > 0$, let
	\[D_\eta(x) = \sup_{0 < \epsilon < \eta} \log \left( \frac{a(x + \epsilon) - a(x)}{\epsilon} \right) . \]
These functions are uniformly bounded, because we know that $a$ is uniformly Lipschitz on the $\RR$-orbits \csee{LipOnLeaves}. Also, the functions converge pointwise to $\dleaf_a$ as $\eta \to 0$. By the Dominated Convergence Theorem, then
	\begin{align} \label{chiPLocalPf-toChi}
	 \textstyle \int_X D_\eta \, d\mu_X^P \stackrel{\eta \to 0}{\longrightarrow} \chi_P(a) 
	 . \end{align}

Now, a technical point. We will apply the Pointwise Ergodic Theorem~\pref{PointwiseErgThm} to the action of~$a$. This requires an ergodic measure. We know from \cref{muXPErgodic} that the measure $\mu_X^P$ is ergodic for the $P$-action, but we cannot expect it to be ergodic for the action of~$a$. Therefore, we should consider the ergodic components of the measure. Fortunately, it can be shown that the integral in \cref{chiDefn} that defines~$\chi_P(a)$ has the same value for each ergodic component \cite[Prop.~7.2]{DeroinHurtado}, so this is not a serious problem. Therefore, we will simply ignore it, and act as if the measure $\mu_X^P$ is ergodic for the action of~$a$.

Note that, for $\epsilon, \delta > 0$, $\eta$ small, and $n$~large, we have
	\begin{align*}
	\frac{1}{n} \log \left( \frac{ a^n (x + \epsilon) - a^n x }{\epsilon} \right)
	&= \frac{1}{n}  \bigl[ \log \bigl( a^n (x + \epsilon) - a^n x \bigr] - \log \epsilon \bigr]
	\\&= \frac{1}{n}  \sum_{i=1}^n \log \left( \frac{ a^i (x + \epsilon) - a^i x }{a^{i-1} (x + \epsilon) - a^{i-1} x } \right)
		&& \text{(telescoping sum)}
	\\&\le \frac{1}{n} \sum_{i=1}^n D_\eta(a^{i-1}x)
		&& \begin{pmatrix} \text{if $a^{i-1} (x + \epsilon) - a^{i-1} x < \eta$}\\ \text{for $1 \le i < n$}\end{pmatrix}
	\\&< \delta + \int_X D_\eta \, d\mu_X^P
		&& \text{(Pointwise Ergodic Theorem~\pref{PointwiseErgThm})}		
	\\&< 2\delta + \chi_P(a)
		&& \text{\pref{chiPLocalPf-toChi}}
	. \end{align*}
So
	\[ \text{$a^n (x + \epsilon) - a^n x \ < \ \epsilon \, e^{n c }$ \quad where $c = 2\delta + \chi_P(a)$} . \]
We can choose $\delta$ small enough that $2\delta + \chi_P(a) < 0$, so $e^{nc} \tends 0$. Then we may choose $\epsilon$ small enough that $e^{n c} < \eta$ for all~$n$. Then we conclude, by induction, that $a^n (x + \epsilon) - a^n x < \eta$ for all~$n$. So the above estimate applies for all~$n$. Hence, we have shown that $a^n (x + \epsilon) - a^n x \tends 0$ (exponentially fast).

A similar argument with negative $\epsilon$ completes the proof.
\end{proof}

\subsection{\texorpdfstring{$\mu_X^P$}{The measure on X} is \texorpdfstring{$G$}{G}-invariant}

\begin{rem}[further details for the proof of \cref{C_G(a)}] \label{C_G(a)Details}
The proof of \cref{C_G(a)} strongly uses the fact that $\mu_X^P$ is $U_a^+$-invariant, and also uses the fact that $\pi_* (c_*\mu_X^P)$ is a ${}^c\!P$-invariant probability measure, but the latter is needed only to conclude that the support of this measure is all of $G/\Gamma$.

Fortunately, both of these properties are true of the ergodic components. Firstly, it follows from the Mautner Lemma \cite[Lem.~II.3.6, p.~61]{BekkaMayerBook} that if a probability measure is invariant under both~$a$ and its expanding horospherical subgroup~$U_a^+$, then each $a$-ergodic component~$\mu$ of the measure is invariant under~$U_a^+$. 
Furthermore, it can be shown (by an argument very similar to the proof of \fullcref{PActsMinimal}{set}) that $G/\Gamma$ has no nonempty, proper, closed subsets that are invariant under $\langle a, U_a^+ \rangle$. 
This implies that the support of $\pi_*\mu$ is all of $G/\Gamma$.
\end{rem}

\begin{cor}[proof of the general case of \cref{propagating}, {\cite[Thm.~8.1 (and~\S9)]{DeroinHurtado}}] \label{GeneralPropagatingPf}
The measure $\mu_X^P$ is $G$-invariant. 
\end{cor}

\begin{proof}
Let $\mathcal{Q}_{\text{good}}$ be the set of all minimal parabolic subgroups~$Q$ containing~$A$, such that $\mu_X^P$ is $Q$-invariant (or, equivalently, such that $\mu_X^Q = \mu_X^P$). Also let 
	\[ \halfspace = \{\, a \in A \mid \chi_P(a) < 0 \,\} . \] 
The arguments given in the ``Idea of proof'' of \cref{propagating} show that if $Q_1 \in \mathcal{Q}_{\text{good}}$, and $\weyl_{Q_2}$ is a Weyl chamber, such that $\halfspace$ contains some nonidentity element of $\weyl_{Q_1} \cap \weyl_{Q_2}$, then $Q_2 \in \mathcal{Q}_{\text{good}}$. Therefore, if we let 
	\[ \agood \ \coloneqq \bigcup_{Q \in \mathcal{Q}_{\text{good}}} \weyl_Q , \]
it is not difficult to show that $\halfspace \subseteq \agood$. 

Furthermore, if $w$ is any element of the Weyl group, such that $w(\halfspace) \cap \agood$ is nonempty, then one can repeat the argument with $w(\halfspace)$ in the place of~$\halfspace$, to conclude that $\agood$ also contains~$w(\halfspace)$. For every element~$w$ of the Weyl group, this implies that either $w(\agood) = \agood$ or $w(\agood)$ is disjoint from the interior of~$\agood$.

\refstepcounter{caseholder}

\begin{case}
Assume $\agood$ is invariant under the Weyl group.
\end{case}
Since the Weyl group acts transitively on the set of Weyl chambers, this implies that every Weyl chamber is in~$\agood$, so $\agood = A$. This means that every minimal parabolic subgroup is in~$\mathcal{Q}_{\text{good}}$, so $\mu_X^P$ is invariant under every minimal parabolic subgroup. Since the minimal parabolic subgroups generate~$G$, we conclude that $\mu_X^P$ is $G$-invariant, as desired.

\begin{case} \label{GeneralPropagatingPf-hard}
Assume $\agood$ is \ul{not} invariant under the Weyl group.
\end{case}
Recall that $\agood$ contains the halfspace~$\halfspace$, and also that, for every element~$w$ of the Weyl group, either $w(\agood) = A$, or $w(\agood)$ is disjoint from the interior of~$\agood$. Then the assumption of this case implies that $\agood = \overline{\halfspace}$ (the closure of~$\halfspace$) and, for every element of the Weyl group, either $w(\overline{\halfspace}) = \overline{\halfspace}$ or $w(\overline{\halfspace}) = \overline{\halfspace^{\text{op}}}$, the closure of the opposite halfspace $\{\, a \in A \mid \chi_P(a) > 0 \,\}$. 

The action of the Weyl group of a simple group is irreducible, so the only way this can happen is if:
	\begin{itemize}
	\item $G$ has a simple factor~$G_1$ with $\rank_{\QQ_p} G_1 = 1$,
	and
	\item $\chi_P = c \, \alpha$, where $\alpha$ is the (unique) simple root of~$G_1$ (for some ordering of the roots), and $c \in \RR^+$.
	\end{itemize}

Write $G = G_1 \times H$, and choose some $a \in A \cap G_1$, such that $\alpha(a) < 0$. Then $a$ centralizes~$H$, so we see from \cref{C_G(a)} that $\mu_X^P$ is $H$-invariant. The subgroup~$H$ is normal, so this implies that $k_* \mu_X^P$ is also $H$-invariant, for every $k \in K$:
	\[ h_* (k_* \mu_X^P) = k_* (k^{-1} h k)_* \mu_X^P = k_* \mu_X^P . \]
Since $\mu_X$ is the integral of these $H$-invariant measures \csee{FurstenburgMuP}, we conclude that $\mu_X$ is $H$-invariant. 
And we know that $\mu_X$ is $K$-invariant (by definition). So $\mu_X$ is invariant under $KH$. Thus, letting $\Gamma_0 \coloneqq \Gamma \cap KH$, we know that $\mu_X$ is invariant under~$\Gamma_0$. For every $\gamma \in \Gamma$, this implies that $\dleaf_\gamma = 1$ $\mu_X$-a.e.\ \csee{RadonNikodym}, so $\gamma$ commutes with the $\RR$-flow. Therefore, $\Gamma_0$ acts by translations on each $\RR$-orbit in~$Z$. Since the group of translations is abelian, this implies that the commutator subgroup $[\Gamma_0, \Gamma_0]$ acts trivially on~$Z$. On the other hand, $\Gamma$ acts faithfully on~$Z$. So we conclude that $\Gamma_0$ is abelian.

Let $\pi \colon G \to G_1$ be the projection with kernel~$H$.
Since $\Gamma$ is irreducible, we know that $\pi(\Gamma)$ is dense in~$G_1$. Since $K$ is open, this implies that $\pi(\Gamma_0)$ is dense in $\pi(K)$, which is an open subgroup of~$G_1$. Since nontrivial open subgroups are Zariski dense, we conclude that $\pi(\Gamma_0)$ is Zariski dense in~$G_1$. Since $G_1$ is not abelian, this implies that $\pi(\Gamma_0)$ is not abelian, which contradicts the conclusion of the preceding paragraph.
\end{proof}

\begin{rem}
For real groups, \cref{GeneralPropagatingPf-hard} of the proof of \cref{GeneralPropagatingPf} is much longer \cite[\S9, pp.~55--64]{DeroinHurtado}. This is because $K$ is not open, so the intersection $\Gamma \cap K H$ may be abelian, or could even be trivial, so a different idea is needed.
\end{rem}

\subsection{More proof of \texorpdfstring{\cref{keyprop}}{the Key Proposition}} \label{globalSect}

\begin{prop}[proof of \cref{EveryActContractsOverview}, {\cite[Prop.~6.3]{DeroinHurtado}}] \label{EveryActContracts}
For every action of\/~$\Gamma$ on~$\RR$ without fixed points, there is a compact subset~$I$ of\/~$\RR$, such that for every $\epsilon > 0$ and all $x,y \in \RR$, there exists $\gamma \in \Gamma$, such that
	\[ \text{$|\gamma x - \gamma y| < \epsilon$ \ and \ $\{\gamma x, \gamma y\} \subset I$}  . \]
\end{prop}

\begin{proof}
It is a general fact about orientation-preserving actions of finitely generated groups on the real line that one of the following possibilities must hold \cite[Thm.~7.1]{DeroinEtAl-Symmetric}:
	\begin{enumerate}
	\item \label{EveryActContracts-fp}
	the action has a fixed point, or
	\item \label{EveryActContracts-abel}
	there is a nontrivial homomorphism $\Gamma \to \RR$, or
	\item \label{EveryActContracts-S1}
	there is a fixed-point-free homeomorphism~$\varphi$ of~$\RR$ that commutes with every element of~$\Gamma$, or
	\item \label{EveryActContracts-good}
	the action has the desired contraction property.
	\end{enumerate}
Since we are assuming that~\pref{EveryActContracts-fp} does not hold, and \fullcref{SSLattFacts}{abel} tells us that~\pref{EveryActContracts-abel} does not hold, we may assume that~\pref{EveryActContracts-S1} holds.

We may assume that $\varphi$ is the translation $\varphi(t) = t + 1$.
Then $\Gamma$ acts on $\RR / \ZZ = S^1$. By \fullcref{SSLattFacts}{CircleAction}, this action has a finite orbit, so, by passing to a finite-index subgroup of~$G$, we may assume that it has a fixed point. Assuming, without loss of generality, that the fixed point is $0 + \ZZ$, this implies that the set $\ZZ$ is $\Gamma$-invariant. Now, since the action of~$\Gamma$ is orientation-preserving, the action on~$\ZZ$ must be by translations. However, the group of translations is isomorphic to the abelian group~$\ZZ$, whereas the abelianization of~$\Gamma$ is finite (see \fullcref{SSLattFacts}{abel} again). Hence, every point in~$\ZZ$ is fixed by~$\Gamma$, which contradicts an assumption of this \lcnamecref{EveryActContracts}.
\end{proof}

\begin{prop}[proof of \cref{RandomContractOverview}, {\cite[Props.~6.4 and~6.8]{DeroinHurtado}}] \label{RandomContract}
There is a constant $\delta < 1$, such that, for $\mu_X$-a.e.~$x \in X$, for all $y \in \RR x$, and for $\mu_G^\infty$-a.e.\ sequence $(g_i)_{i=1}^\infty$ of elements of~$G$, we have
	\[ \text{$|g_n g_{n-1} \cdots g_1 x \ - \ g_n g_{n-1} \cdots g_1 y| < \delta^n$ \quad for all sufficiently large~$n$} . \]
\end{prop}

\begin{proof}
For convenience, let $\pi_n = g_n g_{n-1} \cdots g_1$. For any $\ell > 0$ and $x \in X$, it follows from \cref{EveryActContracts} that if $n$ is sufficiently large, then there is a positive probability that $\pi_n(x + \ell) - \pi_n x$ is small. Then an argument similar to the proof of \cref{chiPLocal} shows that this difference almost surely tends to~$0$. Hence, we see from the compactness of~$X$ that, for every $\ell, \epsilon > 0$, there exist $p > 0$ and $N \in \NN$, such that, for every $x \in X$, we have
	\[ \mu_G^\infty \bigl( \{\, (g_i)_{i=1}^\infty \mid  \pi_N(x +\ell) - \pi_N x \le \epsilon\bigr\} \bigr) \ge p . \]
For $x \in X$, let
	\[ C_x = \{\, g \in G \mid g(x + \ell) - g(x - \ell) < \epsilon \,\} , \]
so $\mu_G^{*N}(C_x) \ge p$ (where $\mu_G^{*N}$ is the $N$-fold convolution $\mu_G * \mu_G * \cdots * \mu_G$).

We claim for each $x \in X$ and $\mu_G^\infty$-a.e.\ sequence $(g_i)_{i=1}^\infty$ of elements of~$G$, there are infinitely many~$n$, such that $\pi_{n + N} \pi_n^{-1} \in C_{\pi_n x}$. To see this, let $c(x)$ be the probability that this happens at least once. Note that:
	\begin{itemize}
	\item if $\pi_N \in C_x$, then we may let $n = 0$,
	and
	\item otherwise, we may ignore $n = 0,1,2,\ldots, N-1$, and start our search at $n = N$.
	\end{itemize}
Therefore
	\[ c(x) \ge \mu_G^{*N}(C_x) + \int_{G \smallsetminus C_x} c (gx) \, d \mu_G^{*N}(g) . \]
Therefore, the infinum $c_-$ of $c(x)$ satisfies
	\[ c_- \ge p + (1 - p) c_- = c_- + (1 - c_-) p , \]
so $c_- = 1$. Hence, we have $c(x) = 1$ for all~$x$. Then, since the random walk is a Markov process, we conclude that the event happens an infinite number of times.

Now, the Martingale Convergence Theorem tells us that for every $x \in X$, every $t \in \RR$, and $\mu_G^\infty$-a.e.\ $(g_i)_{i=1}^\infty$, the sequence $\pi_n(x + t) - \pi_n(x - t)$ converges to a finite limit. For $\ell \in \RR^+$, let $\mathcal{E}_\ell = \mathcal{E}_\ell(x,t)$ be the set of sequences for which this limit is strictly less than~$\ell$. Then, for $n$ sufficiently large, the length of the interval is less than~$\ell$. Then, by the claim of the preceding paragraph, we know (almost surely) that there are infinitely many~$n$, such that $\pi_{g_N g_{N-1} \cdots g_{n+1}} \in C_{\pi_n x}$. This implies 
	\[ \pi_N(x + t) - \pi_N(x - t) < \epsilon . \]
Hence, since we know that the sequence $\pi_n(x + t) - \pi_n(x - t)$ converges, we conclude that it converges to~$0$.

An argument similar to the proof of \cref{chiPLocal} (but using~\pref{Liap} in place of the assumption that $\chi_P(a) < 0$) shows that the contraction is exponentially fast \cite[Lem.~6.7 and Prop.~6.8]{DeroinHurtado}.
\end{proof}

As mentioned in \cref{OverviewSect}, the next part of the argument is easier for real groups, so we will switch to that setting for the remainder of the \lcnamecref{globalSect}:
	\[ \text{\emph{Assume, contrary to \cref{pGGamma}, that $G$ is a connected real semisimple Lie group.}} \]

\begin{rem}[a technical point in the proof of \cref{ExistsContractingOverview}] \label{ExistsContractingDetails}
For convenience, let $a = \widehat{a_P}$. Also let 
	\[ X_0 = \{\, x \in X \mid \text{for all $u \in x + \RR$, we have $\widehat{a_P}^n \, x \  - \ \widehat{a_P}^n \, y \  \tends \ 0 \,\}$} .\]
The ``Idea of proof'' of \cref{ExistsContractingOverview} established that $\mu_X(X_0) = 1$, and we wish to show that $\mu_X^P(X_0) = 1$.

The measure $\mu_X$ is $K$-invariant, so $k_*\mu_X(X_0) = 1$ for all $k \in K$. By Fubini's Theorem, we conclude, for $\mu_X$-a.e.\ $x \in X$, that we have $kx \in X_0$ for $m_K$-a.e.\ $k \in K$. Since $\mu_X = \int_K k_* \mu_X^P \, dk$, we can change $\mu_X$ to~$\mu_X^P$:
	\begin{align} \label{ExistsContractingDetails-Kinvt}
	 \text{for $\mu_X^P$-a.e.\ $x \in X$, for $m_K$-a.e.\ $k \in K$, we have $kx \in X_0$}
	 . \end{align}

Fix $\mu_X^P$-a.e.\ $x \in X$. Also fix a neighbourhood $U$ of~$\1$ in~$U^+_a$, and let $\open^-$ be a small neighbourhood of~$\1$ in the set
	\[ \{\, g \in G \mid \text{$\{a^n g a^{-n}$ is bounded} \,\} . \]
We see from~\pref{ExistsContractingDetails-Kinvt} that for $m_{U_a^+}$-a.e.\ $u \in U$, we can choose some $g \in \open^-$, such that $gu x \in X_0$. Letting $g_n \coloneqq a^n g a^{-n}$, we have
	\[ | g_n a^n u x - g_n a^n uy| = | a^n gu x - a^n guy| \tends 0 .\]
Since $\{g_n\}$ is bounded, this implies 
	\[ |a^n u x - a^n uy| \tends 0 . \]
So $ux \in X_0$. This (and Fubini's Theorem) establishes that $u_* \mu_X^P(X_0) = 1$ for a.e.\ $u \in U$. However, $\mu_X^P$ is $U$-invariant (since $U \subseteq U^+_a \subseteq P$). So we conclude that $\mu_X^P(X_0) = 1$.
\end{rem}

We can also show that  $\chi_P$ is nontrivial. More precisely:

\begin{prop}[proof of \cref{NotVacuous}, {\cite[Lem.~7.4]{DeroinHurtado}}] \label{ChiNontriv}
For $\widehat{a_P}$ as in \cref{trackingA}, we have $\chi_P(\widehat{a_P}) < 0$ and $U_{\widehat{a_P}}^+ \subseteq P$.
\end{prop}

\begin{proof}[Idea of proof]
Since $\widehat{a_P} \in \weyl_P$, we know that $U_{\widehat{a_P}}^+ \subseteq P$. Therefore, we only need to prove the other statement.

The Pointwise Ergodic Theorem~\pref{PointwiseErgThm} has an analogue for random walks with a stationary measure \cite[Cor.~4.8]{DeroinHurtado}. (This was also mentioned in the proof of \cref{muXPErgodic}.) Applying this (and the Chain Rule) to the function $\log \dleaf_g(x)$ on $G \times X$ yields: for $\mu_G^\infty$-a.e.\ sequence $(g_i)_{i=1}^\infty$ of elements of~$G$, and $\mu_X$-a.e.\ $x \in X$, if we let $\pi_n = g_n g_{n-1} \cdots g_1$, then
	\begin{align} \label{ChiNontrivPf-g}
	 \frac{1}{n} \log \dleaf_{\pi_n}(x) 
	=  \frac{1}{n} \sum_{i =0}^{n-1} \log \dleaf_{g_{i+1}} (\pi_i x) 
	\tends \int_G \int_X \log\dleaf_g \, d\mu_X(x) \, d\mu_G(g)
	< 0
	, \end{align}
where the final inequality is from~\pref{Liap}.

For convenience, let $a = \widehat{a_P}$. 
As in the proof of \cref{C_G(a)}, there is a technical issue that the measure $\mu_X^P$ might not be ergodic for the action of~$a$, even though \cref{muXPErgodic} tells us that it is ergodic for the action of~$P$. An additional technical issue is that the above inequality is valid for $\mu_X$-a.e.~$x$, but the definition of $\chi_P(\widehat{a_P})$ uses the measure~$\mu_x^P$. We will ignore both of these issues. (Roughly speaking, they are handled by proving that statements we prove for~$x$ remain valid for~$kx$, for a.e.\ $k \in K$.)

Assuming that $\mu_X^P$ is ergodic with respect to~$a$, we have, for $\mu_X^P$-a.e.\ $x \in X$:
	\begin{align} \label{ChiNontrivPf-a}
	\frac{1}{n} \log \dleaf_{a^n}(x) 
	&= \frac{1}{n}  \sum_{k=0}^{n-1} \log \dleaf_a(a^k x) 
		&& \text{(Chain Rule)}
	\\&\tends \int_X \log \dleaf_a \, d\mu_X^P                \notag
		&& \text{(Pointwise Ergodic Theorem~\pref{PointwiseErgThm})}
	\\&= \chi_P(a)                \notag
		&& \text{(definition of~$\chi_P$)}
	. \end{align}
Also, it follows from \cref{trackingA} that 
	\[ \frac{\log \dleaf_{a^n}(x) - \log \dleaf_{\pi_n}(x)}{n} \tends 0 .\]
Therefore, if we ignore the technical issue that~\pref{ChiNontrivPf-a} is for $\mu_X^P$-a.e.~$x$, but~\pref{ChiNontrivPf-g} is for $\mu_X$-a.e.~$x$, we can equate the two (in the limit), and conclude that $\chi_P(a) < 0$.
\end{proof}

\section{\texorpdfstring{$K\Gamma$}{KΓ} might not be all of~\texorpdfstring{$G$}{G}} \label{FiniteSetSect}

We have been assuming that $G = K \Gamma$ \csee{G=KGamma}. This makes it easy to extend any function~$f$ on~$\Gamma$ to a function~$f_G$ on~$G$, simply by making it left $K$-invariant: $f_G(k \gamma) = f(\gamma)$. Without this assumption, there is usually no canonical way to extend a function on~$\Gamma$ to a function on~$G$. However, we can do this for ``harmonic'' functions.

\begin{defn}[Furstenberg {\cite[Defn.~4.1]{Furstenberg-Poisson}}]
Assume $\mu_H$ is a probability measure on a topological group~$H$, 
	and
	 $f \colon H \to \RR$.
We say that $f$ is \emph{$\mu_H$-harmonic} if, for all $x \in H$, we have
	\[ f(x) = \int_H f(h x) \, d\mu_H(h) . \]
That is, the value of~$f$ at any point is equal to the average (with respect to~$\mu_H$) of the values of~$f$ at certain other points.
\end{defn}

\begin{lem}[{\cite[p.~363]{Furstenberg-Poisson}}] \label{HarmInvt}
Every $\mu_G$-harmonic function~$f$ on~$G$ is left $K$-invariant. Therefore, $f$ corresponds to a well-defined function~$\overline f$ on $K \backslash G$.
\end{lem}

\begin{proof}
For $k \in K$, we have
	\begin{align*}
	f(k x)
	&= \int_G f(g k x) \, d\mu_G(g) 
		&& \text{($f$ is $\mu_G$-harmonic)}
	\\&= \int_G f(g x) \, d\mu_G(g) 
		&& \text{($\mu_G$ is right $K$-invariant \fullcsee{nuGDefn}{G})}
	\\&= f(x) 
		&& \text{($f$ is $\mu_G$-harmonic)}
	. \qedhere \end{align*}
\end{proof}

\begin{rem} \label{HarmDetByGamma}
Choose some $x_0 \in G$, and construct the random walk $x_0, x_1, x_2,\ldots$ in~$G$, starting at~$x_0$, with law~$\mu_G$. (This means that, for every $n \ge 1$, the distribution of the $n$th~step $x_n x_{n-1}^{-1}$ is given by the measure~$\mu_G$.) 
We define a stopping time
	\[  N = N(x_0, x_1, \ldots) = \min\{\, n \ge 0 \mid x_n \in K \Gamma \,\} . \]

Now, suppose $f \colon G \to \RR$ is $\mu_G$-harmonic. Then the sequence $f(x_0), f(x_1), f(x_2), \ldots$ is a \emph{martingale}. More precisely, $f(x_n)$ is the expected value of~$f(x_{n+1})$, given $x_0, x_1, \ldots, x_n$. So it is not difficult to see that $f(x_0)$ is the expected value of $f(x_N)$ (where $N$ is the above-mentioned stopping time, or any other stopping time that satisfies mild conditions). Since $x_N \in K \Gamma$ and $f$ is left $K$-invariant \csee{HarmInvt}, we can conclude that $f(x_0)$ is determined by the values of~$f$ on~$\Gamma$. Since $x_0$ is arbitrary, this means that the value of~$f$ at every point of the group~$G$ is determined by its values on the subgroup~$\Gamma$. Of course, this argument relies on the assumption that $f$ is $\mu_G$-harmonic.
\end{rem}

Conversely, given an appropriate function on~$\Gamma$, we will explain how to extend it to a $\mu_G$-harmonic function that is defined on all of~$G$.

\begin{defn} \label{RandWalkDefns}
\leavevmode\noprelistbreak
	\begin{enumerate}
	\item For each $x \in K \backslash G$, we define an atomic probability measure~$\omega_x$ on $K \backslash G$ by:
	\[ \text{$\omega_{x}(y) = \mu_G(Khg^{-1})$
	\qquad if $x = Kg$ and $y = Kh$} . \]
This is clearly independent of the choice of the representative~$h$ of~$y$, and the following calculation shows that it is also independent of the choice of the representative~$g$ of~$x$:
	\[ \mu_G \bigl (Kh (kg)^{-1} \bigr)
		=  \mu_G \bigl (Kh g^{-1} k^{-1} \bigr)
		=  \mu_G \bigl (Kh g^{-1} \bigr) , \]
where the final equality follows from our assumption that $\mu_G$ is right $K$-invariant \fullcsee{nuGDefn}{G-Kinvt}
.
	\item We consider the random walk~$\omega = \omega(x_0)$ on $K \backslash G$ that starts at a point~$x_0$, and has transition probability $\omega_x(y)$ from~$x$ to~$y$.
	\item Let $\widehat x = K$ be the basepoint of $K \backslash G$. For $x \in K \backslash G$ and $\gamma \in \Gamma$, let $\mu_x(\gamma)$ be the probability that $\widehat x \gamma$ is the first point of $\widehat x \, \Gamma$ that is encountered by the random walk $\omega(x)$ that starts at~$x$.
	\item Let $\mu_{K \backslash G}$ be the (atomic) measure on $K \backslash G$ that is induced by~$\mu_G$. More precisely, if $\pi \colon G \to K \backslash G$ is the natural quotient map, then $\mu_{K \backslash G} = \pi_* \mu_G$.
	\item \label{RandWalkDefns-muGamma}
	Define a measure~$\mu_\Gamma$ on~$\Gamma$ by:
		\[ \mu_\Gamma(\gamma) = \int_{K \backslash G} \mu_x(\gamma) \, d\mu_{K \backslash G}(x) . \]
	\end{enumerate}
\end{defn}

We will need an estimate of the following type:

\begin{lem} \label{RandWalkTail}
We have $\mu_x(\Gamma) = 1$ for all $x \in K \backslash G$. 

Indeed, if we let $\ell(\gamma)$ be the word length of~$\gamma$, with respect to some finite generating set of\/~$\Gamma$, then there exists $\delta < 1$, such that 
	\[ \sum_{\ell(\gamma) < n} \mu_x(\gamma) > 1 - \delta^n \qquad \text{for all large~$n$.} \]
\end{lem}

\begin{proof}
The random walk~$\omega$ on $K \backslash G$ factors through to a random walk~$\overline\omega$ on $K \backslash G / \Gamma$: for $x,y \in K \backslash G$, the transition probability from $x \Gamma$ to $y \Gamma$ is $\sum_{\gamma \in \Gamma} \omega_x(y \gamma)$. This is a random walk (or ``Markov chain'') on a finite set \csee{KGGammaFinite}, so there is some $r \in \NN$ and $\epsilon > 0$, such that, for every starting vertex, the walk has a probability $> \epsilon$ of reaching the point $\widehat x \,\Gamma$ within $r$~steps.
	Therefore, the probability that the stopping time~$N$ is greater than~$n$ is less than $(1 - \epsilon)^{n/r}$.

To complete the proof, combine this observation with the fact that, since $\Gamma$ is cocompact, the word length metric~$\ell_\Gamma$ on~$\Gamma$ is quasi-isometric to the metric on $\widehat x \, \Gamma$ that is obtained by restricting the metric on~$K \backslash G$ \cite[Prop.~I.8.19, p.~140]{BridsonHaefligerBook}.
\end{proof}

\begin{prop}[{\cite[Thm.~3.12]{DeroinHurtado}}] \label{ExtendHarm}
Every positive \textup(or bounded\/\textup) $\mu_\Gamma$-harmonic function on~$\Gamma$ extends to a unique $\mu_G$-harmonic function on~$G$. Specifically, the $\mu_G$-harmonic extension~$f_G$ of~$f$ is defined by:
		\[ \overline{f_G}(x) = \sum_{\gamma \in \Gamma} \mu_x(\gamma) \, f(\gamma) 
		\qquad \text{for $x \in K \backslash G$} . \]
\end{prop}

\begin{proof}
To simplify the estimates, we assume that $f$ is bounded.
Then it is clear from \cref{RandWalkTail} that the sum defining $\overline{f_G}(x)$ converges, so $f_G$ is a function from~$G$ to~$\RR$.

The uniqueness of the extension is a consequence of the argument in \cref{HarmDetByGamma} (if the extension exists).

For $x = K \gamma \in \widehat x \, \Gamma$, the measure $\mu_x$ is the point mass at~$\gamma$. So it is immediate from the definition of~$\overline{f_G}$ that $\overline{f_G}(K \gamma) = f(\gamma)$. Therefore the formula for~$f_G$ does provide an extension of~$f$.

To complete the proof, we show that $f_G$ is $\mu_G$-harmonic.

For $x_0 \notin \widehat x \, \Gamma$, we see from the definition that 
	\[ \mu_{x_0}(\gamma) = \int_{K \backslash G} \mu_{x_1}(\gamma) \, d \mu_{K \backslash G}(x_1) , \]
so 
	\[\overline{f_G}(x_0)
		= \sum_{\gamma \in \Gamma} \mu_{x_0}(\gamma) \, f(\gamma) 
		= \int_{K \backslash G} \sum_{\gamma \in \Gamma} \mu_{x_1}(\gamma) \, f(\gamma)  \, d \mu_G(x_1)
		= \int_{K \backslash G}  \overline{f_G}(x_1) \, d \mu_G(x_1)
		,\]
so $f_G$ is $\mu_G$-harmonic at the points of~$G$ that are in~$x_0$. 

To see that $f_G$ is also $\mu_G$-harmonic at points of $\widehat x \, \Gamma$, note that
	\begin{align*}
	\overline{f_G}(K \lambda)
	&= f(\lambda)
		&& \text{($f_G$ is an extension of~$f$)}
	\\&= \sum_{\gamma \in \Gamma} \mu_\Gamma(\gamma) \, f(\gamma \lambda)
		&& \text{($f$ is $\mu_\Gamma$-harmonic)}
	\\&= \int_{K \backslash G} \sum_{\gamma \in \Gamma} \mu_x(\gamma) \, f(\gamma \lambda) \, d\mu_{K \backslash G}(x)
		&& \text{(definition of $\mu_\Gamma$)}
	\\&= \int_{K \backslash G} \sum_{\gamma \in \Gamma} \mu_{x\lambda}(\gamma\lambda) \, f(\gamma \lambda) \, d\mu_{K \backslash G}(x)
		&& \begin{pmatrix} \text{definition of~$\mu_x$, and the} \\ \text{random walk~$\omega$ is $G$-invariant} \end{pmatrix}
	\\&= \int_{K \backslash G} \sum_{\gamma \in \Gamma} \mu_{x\lambda}(\gamma) \, f(\gamma) \, d\mu_{K \backslash G}(x)
		&& \text{(change of variables in the sum)}
	\\&= \int_{K \backslash G} \overline{f_G}(x \lambda) \, d\mu_{K \backslash G}(x)
		&& \text{(definition of $\overline{f_G}$)}
	. \qedhere \end{align*}
\end{proof}

Properties of harmonic functions (including \cref{ExtendHarm}) play key roles in the Deroin-Hurtado proof. 

\begin{eg} \label{GeneralmuX}
Harmonic functions are the foundation of the construction of the measure~$\mu_X$ when $G \neq K \Gamma$.
Let $\prob(Z)$ be the space of probability measures on~$Z$, and define $f \colon \Gamma \to \prob(Z)$ by $f(\gamma) = \gamma_*^{-1} \mu_Z$. Then $f$ is $\mu_\Gamma$-harmonic (because $\mu_Z$ is $\mu_\Gamma$-stationary, and $\mu_\Gamma$ is \emph{symmetric}, i.e., $\mu_\Gamma(\gamma^{-1}) = \mu_\Gamma(\gamma)$). So it extends to a unique $\mu_G$-harmonic function $f_G \colon G \to \prob(Z)$. Also note that $f$ is right $\Gamma$-equivariant, in the sense that $f(\gamma \lambda) = \lambda^{-1}_* f(\gamma$), so the uniqueness implies that $f_G$ is also right $\Gamma$-equivariant. Hence, $f_G$ represents a section of the bundle $\bigl( G \times \prob(Z) \bigr)/ \Gamma \to G/\Gamma$. Integrating this section with respect to the $G$-invariant probability measure on $G/\Gamma$ yields a probability measure~$\mu_X$ on $(G \times Z) / \Gamma = X$. Since $f_G$ is $\mu_G$-harmonic (and the measure on $G/\Gamma$ is $G$-invariant), the measure~$\mu_X$ is $\mu_G$-stationary.
\end{eg}

\begin{rem}
Harmonic functions are left $K$-invariant \csee{HarmInvt}, so the section $f_G$ of \cref{GeneralmuX} is constant on~$K$. If $X \simeq K \times Z$ (as stated in \fullcref{Xbundle}{X}), this implies that $\mu_X = m_K \times \mu_Z$, which agrees with \cref{muZmuXDefn}.
\end{rem}

\begin{rem}[cf.\ {\cite[5.3.2, 5.3.4, and 5.10(3)]{DeroinHurtado}}]
Here is a slightly different way to look at the construction of~$\mu_X$.
Define $D \colon \Gamma \times Z \to \RR^+$ by
	$ D(\gamma, z) = \dleaf_\gamma(z) $.
For $\lambda, \gamma \in \Gamma$, we see from the Chain Rule that 
	\begin{align} \label{muXChainRule}
	 D(\lambda\gamma, z) = D(\gamma, z) \, D(\lambda, \gamma z)  
	 . \end{align}
Since $\mu_Z$ is $\mu_\Gamma$-stationary \csee{nuGammaStationary}, this implies
	\[ \sum_{\lambda \in \Gamma} \mu_\Gamma(\lambda) \, D(\lambda\gamma, z) 
		= D(\gamma, z) \sum_{\lambda \in \Gamma} \mu_\Gamma(\lambda) \,D(\lambda, \gamma z)
		= D(\gamma, z) \cdot 1
		= D(\gamma, z)
		. \]
This means that the function $\gamma \mapsto D(\gamma, z)$ is $\mu_\Gamma$-harmonic, and therefore extends to a unique $\mu_G$-harmonic function on~$G$. Thus, we obtain a function $D_G \colon G \times Z \to \RR^{\ge0}$. (This function is $\mu_G$-harmonic in the $G$-coordinate, and defined for $\mu_Z$-a.e.~$z$.) Let 
	\[ \widetilde{\mu_X} = D_G \cdot (m_G \times \mu_Z) , \]
so $\widetilde{\mu_X}$ is a measure on $G \times Z$ that is in the measure class of $m_G \times \mu_Z$. 

For any $g \in G$, the uniqueness of~$D_G$ implies that~\pref{muXChainRule} still holds if we replace $D$ with~$D_G$, and replace $\lambda$ with~$g$.
Then for $\gamma \in \Gamma$ (acting on the right via $(g,z) \mapsto (g \gamma, \gamma^{-1} z)$), we can use the fact that $d \gamma^{-1}_* \mu_Z/d\mu_Z = D_G(\gamma, z)$ \csee{RadonNikodym} to conclude that
	\[ \frac{ d \gamma_* \widetilde{\mu_X} }{d \widetilde{\mu_X}}(g,z)
		= \frac{D_G(g \gamma^{-1}, \gamma z)}{D_G(g,z)} \cdot \frac{d \gamma_* m_G}{dm_G}(g) \cdot  \frac{d \gamma^{-1}_* \mu_Z}{d\mu_Z} (z)
		= \frac{D_G(g \gamma^{-1}, \gamma z)}{D_G(g,z)} \cdot 1 \cdot D_G(\gamma, z)
		= 1
		. \]
This means that $\widetilde{\mu_X}$ is invariant for the action of~$\Gamma$, so we can construct a well-defined measure~$\mu_X$ on~$X$ from~$\widetilde{\mu_X}$, by restricting the quotient map $G \times Z \to X$ to any fundamental domain of the action.
The harmonicity of~$D$ in the $G$-coordinate implies that $\mu_X$ is $\mu_G$-stationary. 

The advantage of this approach is that it shows $\mu_X$ is in the same measure class as $m_G \times \mu_Z$ (restricted to a fundamental domain).
\end{rem}

\end{document}